\documentclass[11pt,reqno]{amsart}

\usepackage[mathscr]{euscript}
\usepackage{amssymb}
\usepackage{textcase}
\usepackage{graphicx}
\usepackage{color}

\makeatletter
\renewcommand\@biblabel[1]{#1.}
\makeatother
\newtheorem{thm}{Theorem}
\newtheorem{lem}{Lemma}
\newtheorem{prop}{Proposition}

\newtheorem{cor}{Corollary}
\theoremstyle{definition}

\newtheorem{df}{Definition}

\newtheorem{Rem}{Remark}
\theoremstyle{definition}
\newtheorem{ex}{Example}

\renewcommand{\subset}{\subseteq}

\newcommand\mc{\mathcal }

\newcommand\mf{\mathfrak }

\newcommand{\R}{{\mathbb R}}

\newcommand{\ls}{\leq^{\star}}
\newcommand{\on}{\operatorname}

\newcommand{\ce}{\mf c}
\newcommand{\pt}{{\mathcal P}}

\newcommand\A{{\mathcal A}}

\DeclareMathOperator\card{card}

\DeclareMathOperator\supp{supp}
\newcommand{\I}{\mathcal I}
\newcommand{\J}{\mathcal J}
\newcommand{\Z}{\mathcal Z}

\title[Properties of simple density ideals]{Properties of simple density ideals}

\author{Adam Kwela}
\thanks{The first and fourth authors have been supported by the grant BW-538-5100-B482-17.}
\address{Institute of Mathematics, Faculty of Mathematics, Physics and Informatics, University of Gda\'{n}sk, ul.~Wita Stwosza 57, 80-308 Gda\'{n}sk, Poland}
\email{adam.kwela@ug.edu.pl}

\author{Micha{\l} Pop{\l}awski}
\address{Institute of Mathematics,
          University of Silesia, ul. Bankowa 14,
         40-007 Katowice,
         Poland}
         \email{michal.poplawski@us.edu.pl}

\author{Jaros{\L}aw Swaczyna}
\thanks{The research of the third author has been supported by the Polish Ministry of Science and Higher Education (a part of the project “Diamond grant”, 2014-17, No. 0168/DIA/2014/43).}
\address{Institute of Mathematics, \L \'od\'z University of Technology,
W\'olcza\'nska 215, 93-005 \L \'od\'z, Poland}
\email {jswaczyna@wp.pl}

\author{Jacek Tryba}
\address{Institute of Mathematics, Faculty of Mathematics, Physics and Informatics, University of Gda\'{n}sk, ul.~Wita Stwosza 57, 80-308 Gda\'{n}sk, Poland}
\email{jtryba@mat.ug.edu.pl}

\begin{document}
\begin{abstract}
Let $G$ consist of all functions $g \colon \omega \to [0,\infty)$ with $g(n) \to \infty$ and $\frac{n}{g(n)} \nrightarrow 0$. Then for each $g\in G$ the family $\Z_g=\{A\subseteq\omega:\ \lim_{n\to\infty}\frac{\card(A\cap n)}{g(n)}=0\}$ is an ideal associated to the notion of so-called upper density of weight $g$. Although those ideals have recently been extensively studied, they do not have their own name. In this paper, for Reader's convenience, we propose to call them \emph{simple density ideals}. 

We partially answer \cite[Problem 5.8]{kwetry} by showing that every simple density ideal satisfies the property from \cite[Problem 5.8]{kwetry} (earlier the only known example was the ideal $\Z$ of sets of asymptotic density zero). We show that there are $\ce$ many non-isomorphic (in fact even incomparable with respect to Kat\v{e}tov order) simple density ideals. Moreover, we prove that for a given $A\subset G$ with $\card(A)<\mf{b}$ one can construct a family of cardinality $\ce$ of pairwise incomparable (with respect to inclusion) simple density ideals which additionally are incomparable with all $\Z_g$ for $g\in A$. We show that this cannot be generalized to Kat\v{e}tov order as $\Z$ is maximal in the sense of Kat\v{e}tov order among all simple density ideals. We examine how many substantially different functions $g$ can generate the same ideal $\Z_g$ -- it turns out that the answer is either $1$ or $\ce$ (depending on $g$).
\end{abstract}
\maketitle
\section{Introduction}
Denote $\omega:=\{0,1,\ldots\}$. We identify $n$ with the set $\{0,\ldots,n-1\}$. By $\card(A)$ we denote cardinality of the set $A$.
Let us recall some basic definitions and facts about ideals. A family $\mathcal{I}$ of subsets of $\omega$ is called an ideal on $\omega$ if the following conditions hold:
\begin{itemize}
\item $\emptyset \in \mathcal{I}$ and $\omega \notin \mathcal{I}$,
\item $A,B \in \mathcal{I} \Rightarrow A \cup B \in \mathcal{I}$,
\item if $A \subset B$ and $B \in \mathcal{I}$, then $A \in \mathcal{I}.$
\end{itemize}
One of the simplest examples of an ideal on $\omega$ is the family $Fin$ of all finite subsets of $\omega$. In this paper we assume about all considered ideals that $Fin \subset \I$. \\
Now we will recall one of the well-known examples of ideals, namely the classical density ideal. Let us denote the upper density of a set $A \subset \omega $ by:
$$\overline{d}(A):=\limsup_{n \to \infty}\frac{\card{(A \cap n)}}{n}$$ (one can define the lower density $\underline{d}(A)$ in an analogous way). The classical density ideal is one of the form: 
$$
\mathcal{Z}:=\{A \subset \omega \colon \overline{d}(A)=0\}.
$$
This ideal has been deeply investigated in the past in the context of convergence (see e.g. \cite{Fast}, \cite{Fridy}, \cite{Salat} and \cite{Steinhaus}) as well as from the set-theoretic point of view (see e.g. \cite{Farah} and \cite{Just}).

Over many years various ways of generalizing the ideal $\Z$, the density function $\overline{d}$ or statistical convergence (which is a notion of convergence associated to $\Z$) have been considered (see e.g. \cite{1}, \cite{3}). This paper concentrates on one of those generalizations. In the paper \cite{Den} authors proposed the following. Let 
$$G:=\{g \colon \omega \to (0,\infty) \colon\ g(n) \to \infty  \wedge \frac{n}{g(n)} \nrightarrow 0 \}.$$
Fix $g \in G$. For any $A \subset \omega$ define 
$$\overline{d_g}(A):=\limsup_{n \to \infty} \frac{\card{(A \cap n)}}{g(n)},$$
called the upper density of weight $g$.
Now, define a family
$$\mathcal{Z}_g:=\{A \subset \omega \colon \overline{d_g}(A)=0\}.$$
It turns out that $\mathcal{Z}_g$ is an ideal on $\omega$ for any $g \in G$. Although ideals of the form $\Z_g$ have recently been extensively studied (see the next paragraph), they do not have their own name and are usually called ideals associated to upper density of weight $g$. In this paper, for Reader's convenience, we propose to call ideals of the form $\Z_g$ \emph{simple density ideals}. Note that the condition $g(n) \to \infty$ is equivalent to $\mc{Z}_g \neq \{ \emptyset \}$, while $\frac{n}{g(n)} \not\to 0$ is introduced to assure that $\omega \notin \mc{Z}_g$. Another basic observation is that for $g = id_{\omega}$ we have $\mc{Z}_g = \mc{Z}$, so the classical density ideal is a special case of simple density ideal. In paper \cite{Den} authors have studied some structural properties of such ideals as well as inclusions between them. 

One motivation for studying simple density ideals is related to various notions of convergence -- the class of simple density ideals (as well as density functions related to them) has recently been intensively studied in this context. For an ideal on $\omega$ we say that a sequence of reals $(x_n)$ is $\I$-convergent to $x\in\mathbb{R}$ if $\{n\in\omega:\ |x_n-x|\geq\varepsilon\}\in\I$ for each $\varepsilon>0$. For instance, in \cite{5} the authors have studied which subsequences of a given $\Z_g$-convergent sequence, where $g\in G$, are also $\Z_g$-convergent. In \cite{6} the notion of $\Z_g$-convergence has been compared (in the context of fuzzy numbers) with the notion of so-called $\I$-lacunary statistical convergence of weight $g$. Paper \cite{4} is in turn an investigation of a variant of convergence associated to simple density ideals and matrix summability methods. 

Our goal is to investigate further properties of the family $\{ \mc{Z}_g : g \in G \}$. Obviously, we will use some results from \cite{Den} in the proceeding sections. Now we cite just one of them, in order to make thinking about simple density ideals a little bit simpler:
\begin{prop}
For any $f \in G$ there exists such a nondecreasing $g \in G$ that $\mc{Z}_f = \mc{Z}_g$. \qed
\end{prop}
Another easy but useful observation is that $\Z_f = \Z_{\lfloor f \rfloor}$, thus we introduce two additional notations:
$$
H:= G \cap \omega^\omega \mbox{ and } H^\uparrow:= \{ f \in H : f \mbox{ is nondecreasing}\}.
$$

We treat $\mathcal{P}(X)$ as the space $2^X$ of all functions $f:X\to 2$ (equipped with the product topology, where each space $2=\left\{0,1\right\}$ carries the discrete topology) by identifying subsets of $X$ with their characteristic functions. All topological and descriptive notions in the context of ideals on $X$ will refer to this topology. 

The paper is organized as follows. Section \ref{sekrozne} is devoted to investigations of some basic properties of the family $\{\Z_g:g\in G\}$. Namely, we show that $\bigcap_{g \in G} \mathcal{Z}_g= Fin$ and $\bigcup_{g \in G} \mathcal{Z}_g$ is equal to $\{A\subset\omega:\underline{d}(A)=0\}$. Moreover, we give a direct proof of the fact that all $\Z_g$ are not $F_{\sigma}$. This has been shown in \cite{Den}, but with the use of results of Farah from \cite{Farah}. 

Section \ref{sekgenerowanieroznymifcjami} investigates how many substantially different functions $g\in G$ can generate the same ideal $\Z_g$. In turns out that the answer is either $1$ or $\ce$. We provide a suitable characterization. 

Section \ref{sekwstepmnog} is not related to ideals on $\omega$ -- it contains various, rather well-known, observations concerning partially ordered sets and the bounding number $\mf{b}$. We will use them in Section \ref{sekporzadekmiedzyidealami} to show that for a given $A\subset G$ with $\card(A)<\mf{b}$ one can construct a family of cardinality $\ce$ of pairwise incomparable (with respect to inclusion) simple density ideals which additionally are incomparable with all $\Z_g$ for $g\in A$. Moreover, in Section \ref{sekporzadekmiedzyidealami} we construct an antichain in the sense of Kat\v{e}tov order of size $\ce$ among simple density ideals. As a consequence, we obtain $\ce$ many pairwise non-isomorphic ideals $\Z_g$. In addition, we prove that $\Z$ is maximal in the sense of Kat\v{e}tov order among all ideals $\Z_g$. 

Section \ref{sekGD+FD+EU} compares simple density ideals with Erd\H{o}s-Ulam ideals and density ideals in the sense of Farah. We characterize when $\Z_g$ is an Erd\H{o}s-Ulam ideal and provide some examples (also in the context of our results from Section \ref{sekgenerowanieroznymifcjami}). 

Finally, in Section \ref{sekhomogeneity} we partially answer \cite[Problem 5.8]{kwetry} by showing that every increasing-invariant ideal $\I$ (i.e., such ideal that $B\in\I$ and $\card(C\cap n)\leq\card(B\cap n)$ for each $n$ implies $C\in\I$)
fulfills that if $\I\upharpoonright A$ for some $A\subset\omega$ is isomorphic to $\I$, then a witnessing isomorphism can be given by just the increasing enumeration of $A$. As every $\Z_g$ is increasing-invariant, we obtain $\ce$ examples of ideals satisfying the above property. We end with an example of an anti-homogeneous simple density ideal.

\section{Basic properties of simple density ideals}\label{sekrozne}
We start this Section with answering the question which has appeared during the presentation of the results of paper \cite{Den} on Wroc\l aw Set Theory Seminar: what do the sets  $\bigcap_{g \in G} \mathcal{Z}_g $ and $\bigcup_{g \in G} \mathcal{Z}_g$ look like?

It is easy to see that $\bigcap_{g \in G} \mathcal{Z}_g=Fin$. Indeed, take any infinite $A \subset \omega$. Then the function $g$ given by $g(n)=\card (A\cap n)$ is in $G$ and obviously $\card(n \cap A) / g(n) \nrightarrow 0$, hence $A \notin \mc{Z}_g$. 

Now we move to $\bigcup_{g \in G} \mathcal{Z}_g$. We will denote the family of lower density zero subsets of $\omega$ by $ \mathcal{Z}_l:=\{A \subset \omega \colon \underline{d}(A)=0\}$. Our goal is to show that $\mc{Z}_l = \bigcup_{g \in G} \mathcal{Z}_g$. To do so, take any $g \in G$ and such $A \subset \omega$ that $\limsup \frac{\card{(A \cap n)}}{g(n)}=0$. Note that since $\frac{n}{g(n)} \nrightarrow 0$ and $g(n) > 0$ for $n \in \omega$, there exists such an increasing sequence $(n_k)_{k \in \omega}$ of naturals that $(\frac{g(n_k)}{n_k})$ is bounded. Hence,
$$
\frac{\card{(A \cap n_k)}}{n_k}=\frac{\card{(A \cap n_k)}}{g(n_k)}\frac{g(n_k)}{n_k} \xrightarrow{k \to \infty} 0,
$$
which proves that $\liminf \card{(A \cap n)}/n=0$. This means that $A \in \mathcal{Z}_l$. \\
Now, take any $A \in \mathcal{Z}_l$. We infer that there is an increasing sequence $(n_i)$ of naturals, satisfying:
$$
\frac{\card{(A \cap n_i)}}{n_i} \rightarrow 0.
$$ 
Define a function $g \colon \omega \to [0,\infty)$ in the following way:
$$
g(n) := \min \{ n_i : n \leq n_i \}.
$$
Note that $g(n) \rightarrow \infty$, as $(n_i)$ is increasing. Moreover, since $g(n_i ) = n_i $, we have $\frac{n_i}{g(n_i)} = 1$ for any $ i \in \omega$, hence $\frac{n}{g(n)} \nrightarrow 0$. As a consequence, $g \in G$. We will prove that $A \in \mathcal{Z}_g$. Take some $n \in \omega$. Then $g(n) = n_i$ for some $i \in \omega$, $n \to \infty$ implies $i \to \infty$ and
$$
\frac{\card{(A \cap n)}}{g(n)}=\frac{\card{(A \cap n)}}{n_{i}} \leq \frac{\card{(A \cap n_{i})}}{n_{i}} \xrightarrow{i \to \infty} 0 .
$$
Thus, $A \in \mathcal{Z}_g$ and we obtain a complete answer to the question from the beginning of this Section:
\begin{prop}
$\bigcup_{g \in G} \mathcal{Z}_g=\mathcal{Z}_l$ and $\bigcap_{g \in G} \mathcal{Z}_g= Fin$. \qed
\end{prop}
In \cite{Den} the authors proved that all simple density ideals are not $F_{\sigma}$ sets, however they used a bit complicated argument based on the density ideals in the sense of Farah. Below we will give a more direct proof of this fact. We will follow methods used in the proof of \cite[Proposition 3]{DrewLab}.

\begin{prop}\label{Fsigma}
$\mc{Z}_g$ is not an $F_{\sigma}$ set for any $g \in G$.
\end{prop}

\begin{proof}
We can assume that $g \in H^\uparrow$. Recall that we usually identify sets $A \subset \omega$ with their characteristic functions $\chi_A \in 2^\omega$. However, in this proof we will use another notation in order to make it more convenient for the Reader, namely we denote $Z_g = \{ \chi_A : A \in \Z_g \}$. Note that for any $A\subset \omega$ we have that $\chi_{A} \in Z_g$ if and only if 
\begin{equation} \label{wr} \frac{1}{g(n)} \sum_{i=0}^n \chi_{A}(i) \rightarrow 0 \end{equation}
as $n \rightarrow \infty.$
Consider a metric $\rho$ on $Z_g$ given by:
$$\rho(u,v)=\sup_{n \in \omega} \frac{1}{g(n)} \sum_{i=0}^{n} \left |u(i)-v(i) \right|$$
for $u,v \in Z_g.$
We will show the completeness of the metric space $(Z_g,\rho)$.
\newline
Take a Cauchy sequence $(t_m)_{m \in \omega}$ in the space $(Z_g,\rho)$. It follows from the definition of $\rho$ that for any $n \in \omega$ a sequence $(t_m(n))_{m \in \omega}$ is Cauchy in the complete metric space $\{0,1\}$. Then there is $t \in 2^{\omega}$ -- a pointwise limit of $(t_m)_{m \in \omega}.$
We will show that $t$ satisfies the condition $(\ref{wr})$. Obviously, an inequality 
\begin{equation} \label{cs} \sup_{n \in \omega} \left|\frac{1}{g(n)} \left(\sum_{i=0}^n t_{k}(i)-\sum_{i=0}^n t_{m}(i)\right)\right| \leq \sup_{n \in \omega} \frac{1}{g(n)} \sum_{i=0}^n \left|t_{k}(i)-t_{m}(i)\right| \end{equation}
holds for all $k,m \in \omega.$
Since the sequence $(t_m)_{m \in \omega}$ is Cauchy in $(Z_{g},\rho)$, it follows from $(\ref{cs})$ then that the sequence $((\frac{1}{g(n)} \sum_{i=0}^n t_{m}(i))_{n \in \omega})_{m \in \omega}$ is a Cauchy sequence in the Banach space $c_0$. As a consequence, the limit 
$$\lim_{m \to \infty} \left(\frac{1}{g(n)} \sum_{i=0}^n t_{m}(i)\right)_{n \in \omega}=:\alpha $$
is a member of $c_0$.
\newline
Since $(t_m)_{m \in \omega}$ tends pointwise to $t$, we get $\frac{1}{g(n)} \sum_{i=0}^n t_{k}(i) \rightarrow \frac{1}{g(n)} \sum_{i=0}^n t(i)$ as $k \rightarrow \infty$ for any $n \in \omega$. Hence, $(\frac{1}{g(n)}\sum_{i=0}^n t(i))_{n \in \omega}=\alpha \in c_0$, thus $t$ satisfies the condition $(\ref{wr})$ and, consequently, $t \in Z_g$. 
\newline
Now we will show that $\rho(t,t_k) \rightarrow 0$ as $k \rightarrow \infty.$ Fix $\varepsilon>0$ and such $N \in \omega$ that $\rho(t_k,t_m)<\varepsilon$ for all $k,m \geq N$. This implies that 
\begin{equation} \frac{1}{g(n)} \sum_{i=0}^n \left|t_k(i)-t_m(i)\right|<\varepsilon \end{equation}
for all $k,m \geq N$ and $n \in \omega.$ Letting $m \to \infty$, we get:
\begin{equation} \label{kc} \frac{1}{g(n)} \sum_{i=0}^n \left|t_k(i)-t(i)\right| \leq \varepsilon \end{equation}
for all $k \geq N$ and $n \in \omega.$
Thus, we obtain:
$$\rho(t_k,t)=\sup_{n \in \omega} \frac{1}{g(n)} \sum_{i=0}^n \left|t_k(i)-t(i)\right| \leq \varepsilon$$
for all $k \geq N$. Hence, $\rho(t_k,t) \rightarrow 0$ as $k \rightarrow \infty.$
\newline
We proved that $(Z_g , \rho)$ is complete. Since $t_m \xrightarrow{\rho} t$ implies that $(t_m )_{m \in \omega}$ is convergent to $t$ in the usual sense on $2^\omega$, we observe also that $id:Z_g \to 2^\omega$ is continuous, i.e., $Z_g$ is continuously included in $2^\omega$. \\
Now we will prove that $Z_g$ is not an $F_{\sigma}$ set. Suppose on the contrary that $Z_g=\bigcup_{n \in \omega} F_n$ for some closed sets $F_n \subset 2^\omega, n \in \omega.$ Since $Z_g$ is continuously included in $2^\omega$, the sets $F_n$ are also closed in $Z_g$, thus $Z_g=\bigcup_{n \in \omega} F_n$ is a countable union of its closed subsets (with respect to $\rho$). Using the Baire Category Theorem, we get an existence of such $n_0 \in \omega$ that $F_{n_0}$ contains a ball $B(u,r)$ for some $u \in Z_g$ and $r>0$. 
Thanks to $(\ref{wr})$ we may fix such $N \in \omega$ that:
$$\frac{1}{g(n)} \sum_{i=0}^n u(i)<\frac{r}{4}$$
for all $n\geq N$.
Observe that if for some $n>N$ and $t(N+1),\ldots,t(n) \in \{0,1\}$ it is true that:
$$\frac{1}{g(n-N)} \sum_{i=N+1}^n t(i) <\frac{r}{4},$$
then the sequence $(w(j))_{j \in \omega}:=(u(0),\ldots,u(N),t(N+1),\ldots,t(n),0,\ldots)$
satisfies:
$$\frac{1}{g(n)} \sum_{i=0}^n |u(i)-w(i)|=\frac{1}{g(n)} \sum_{i=N+1}^n |u(i)-t(i)| $$
$$\leq \frac{1}{g(n)} \sum_{i=0}^n u(i)+\frac{1}{g(n)} \sum_{i=N+1}^n t(i)<\frac{r}{4}+\frac{1}{g(n-N)}\sum_{i=N+1}^n t(i)<\frac{r}{2},$$
since $g$ is nondecreasing. Now we define $N_i$ for $i\in \omega$ by: 
$$ 
N_0 = \on{min} \{ n > N : \frac{1}{g(n-N)} < \frac{r}{4} \} \mbox{ and } N_{i+1} = \on{min} \{ n > N_i : \frac{i+2}{g(n-N)} < \frac{r}{4} \} .
$$
Then clearly $\frac{i+1}{g(N_i-N)}<\frac{r}{4}$ for all $i \in \omega$ and by $n/g(n) \not\rightarrow 0$ we also get $\limsup_{i\to\infty} \frac{i+1}{g(N_i)}>0$. Indeed, suppose that $\lim_{i \to \infty} \frac{i+1}{g(N_i)}=0$ and note that inequality $N_{i}-N>N_{i-N}$ for each $i>N$ (since $N_{i-N}<N_{i-N+1}<\ldots<N_i$ and $N_k \in \omega$ for $k \in \omega$) implies $\frac{i+1}{g(N_i-N)} \leq \frac{i+1}{g(N_{i-N})}=\frac{i-N+1}{g(N_{i-N})}+\frac{N}{g(N_{i-N})}$ and finally $\lim_{i \to \infty} \frac{i+1}{g(N_i-N)}=0$. Now, take such $i_0 \in \omega$ that $\frac{i+1}{g(N_i-N)}<\frac{r}{8}$ for all $i \geq i_0$ (then also $\frac{1}{g(N_i - N)}< \frac{r}{8}$). Therefore,
$$\frac{i+2}{g(N_i+1-N)}=\frac{i+1}{g(N_i+1-N)}+\frac{1}{g(N_i+1-N)} \leq \frac{i+1}{g(N_i-N)}+\frac{1}{g(N_i-N)}<\frac{r}{8}+\frac{r}{8}=\frac{r}{4},$$ 
thus $N_{i+1} \leq N_i +1.$ However, this means that for $A:=\{ N_i : i \geq i_0\}$ we get: 
$$
\frac{\card(A \cap N_i)}{g(N_i)} \leq \frac{i}{g(N_i)} \to 0,
$$
and thus $A \in \Z_g$ -- a contradiction, since the set $\omega \setminus A$ is finite. \\
Now for any $i \in \omega$ we define $t_i \in \{0,1\}^{N_i - N}$ by $t_i := \chi_{\{N_k - N : k \leq i \}}$. Finally, we set:
$$w_i = (u(0), \ldots , u(N), t_i(0), \ldots, t_i(N_i - N - 1), 0, \ldots ).$$
\newline
Then for each $i \in \omega$ we get $w_i\in B(u,r)\subset F_{n_0}$. Since $F_{n_0}$ is closed in $2^\omega$, we obtain that $w:=\lim_{i \to \infty} w_i$ is a member of $F_{n_0} \subset Z_g.$ However, the sequence $(\frac{1}{g(n)}\sum_{j=0}^n w(j))_{n \in \omega}$ is not converging to $0$, because:
$$\frac{1}{g(N_i)} \sum_{j=0}^{N_i} w(j) \geq \frac{i}{g(N_i)}$$ and we know that $\limsup_{i\to\infty} \frac{i}{g(N_i)}>0$. Condition $(\ref{wr})$ shows that $w \notin Z_g$ -- a contradiction, since $w \in F_{n_0} \subset Z_g$. 
\end{proof}

\section{Generating the same ideal by various functions}\label{sekgenerowanieroznymifcjami}

Now we move to another question -- namely, how many functions $g \in G$ satisfy $\mathcal{Z}_g=\mathcal{Z}_f$ for a given $f \in G.$ In our first approach the answer is easy thanks to \cite[Proposition 2.1]{Den}:
\begin{prop}\label{rown}
Given $f,g \in G$, if there are such positive reals $m$ and $M$ that $m \leq \frac{f(n)}{g(n)} \leq M$ for all $n \in \omega$, then $\mathcal{Z}_f=\mathcal{Z}_g$. \qed
\end{prop}
As a result, for a given $f \in G$ we have $\mathcal{Z}_f=\mathcal{Z}_{af}$, where $a$ is any positive number. Hence, for a given $f \in G$ we obtain $\ce$ many functions which generate the same ideal. On the other hand, we have that $\card (G) \leq \card( \R^{\omega}) = \ce$, so the above result is optimal. However, this answer is a little bit of "cheating" -- as each $\Z_f$ is generated by some function from $\omega^\omega$, one could require that all considered functions have values in $\omega$. There are also other ways of "cheating", like considering functions which are not monotone. However, the family $\{ \lfloor \alpha f \rfloor : \alpha \in [1,2] \}$ allows us to omit all those issues as all of its members are in $H^\uparrow$ (recall that $H^\uparrow = \{f \in G : f \in \omega^\omega \mbox{ and } f \mbox{ is nondecreasing}\}$) and $\mathcal{Z}_f=\mathcal{Z}_{\lfloor \alpha f \rfloor}$ for all $\alpha \in [1,2]$. Thus, it seems that the proper question is connected with avoiding Proposition \ref{rown}. In order to ask it in precise way we introduce some notation. For $f\in H^\uparrow$ let $S_f=\left\{g\in H^\uparrow\ : \Z_f=\Z_g\right\}/R$, where $R$ is a equivalence relation given by:
$$gRh\ \Longleftrightarrow\ \limsup_{n\to\infty}\frac{g(n)}{h(n)}<\infty\ \wedge\ \limsup_{n\to\infty}\frac{h(n)}{g(n)}<\infty.$$ 
Now we are interested in the following question: given $f \in H^\uparrow$, what is the cardinality of the set $S_f$?

Our answer consists of two elements. Before we give them, we need two easy lemmata.
\begin{lem}\label{Michal}
For any $f \in H^\uparrow$ and $A = \{ a_0 <a_1 < \ldots\} \subset \omega$ the following equivalence is true: 
$$
A \in \Z_f \Longleftrightarrow \frac{\card(A \cap a_k)}{f(a_k)} \to 0.
$$
\end{lem}
\begin{proof}
Implication from left to right is obvious, so we will focus on the second one. Fix $A = \{ a_0 <a_1 < \ldots\} \subset \omega$ and take $n \in \omega$ with $n\geq a_0$. Then $n \in [a_k , a_{k+1}) $ for exactly one $k \in \omega$ and thus:
$$
\frac{\card(A \cap n)}{f(n)} \leq \frac{\card(A \cap a_{k+1})}{f(a_k)} =
\frac{\card(A \cap a_{k})+1}{f(a_k)} \to 0,
$$
so we are done. 
\end{proof}
\begin{lem} \label{max}
Suppose that $\Z_f=\Z_g$ for some $f,g \in H^{\uparrow}$. Then $\Z_{\max \{f,g\}}=\Z_f$.
\end{lem}
\begin{proof}
Take any $A \subset \omega$. The inequality $\frac{\card(A \cap n)}{\max \{f(n),g(n)\}} \leq \frac{\card(A \cap n)}{f(n)}$ shows that $\Z_f \subset \Z_{\max \{f,g\}}.$ Now suppose that $A \in \Z_{\max \{f,g\}}$ and denote $A_1:=\{n \in A \colon g(n) \leq f(n)\}=\{a_0<a_1< \ldots\}$ (the case with $A_1$ finite is easier) and $A_2:=A \setminus A_1.$ 
Note that $A_1,A_2 \in \Z_{\max \{f,g\}}.$ Then $\frac{\card(A_1 \cap a_k)}{f(a_k)}=\frac{\card(A_1 \cap a_k)}{\max \{f(a_k),g(a_k)\}} \to 0$, and $A_1 \in \Z_f$, thanks to Lemma \ref{Michal}. Similarly, $A_2 \in \Z_g=\Z_f$ and consequently $A=A_1 \cup A_2 \in \Z_f$.
\end{proof}
The first ingredient of our answer to the question about cardinality of $S_f$ is the following Proposition:
\begin{prop}\label{moce2}
Let $f\in H^\uparrow$. If $\card(S_f)> 1$, then $\card(S_f)=\ce$.
\end{prop}
\begin{proof}
Suppose we have such $g\in S_f$ that $f\not=g$ and let us denote $a:= \limsup_{n\to\infty} \frac{n}{f(n)} >0$. Without loss of generality, we may assume that $\limsup_{n \to \infty} \frac{g(n)}{f(n)} = \infty$, thus there exists such a sequence $(n_k)_{k\in\omega}$ that $\lim_{k\to\infty}\frac{g(n_k)}{f(n_k)}=\infty$. By an easy inductive construction, we may also assure that for all $k\in\omega$ we have $f(n_{k+1})\geq g(n_k).$ Now, for each $\alpha=(\alpha_1,\alpha_2,\ldots)\in 2^\omega$ we define a function $f_\alpha$ as follows:
$$f_\alpha(n)=\left\{\begin{array}{ll}
f(n) & \textrm{if } n<n_0,\\
f(n) & \textrm{if  } n\in[n_k,n_{k+1}) \textrm{ and } \alpha_k=0,\\
\max\{g(n_k),f(n) \} & \textrm{if } n\in[n_k,n_{k+1}) \textrm{ and } \alpha_k=1.
\end{array}\right.$$
It is easy to see that for all $\alpha$ we have $\Z_{f_\alpha}=\Z_f$, because $f_\alpha(n)\in [ f(n), \max \{ f(n),g(n) \}]$ for all natural $n$ and $\Z_f=\Z_{\max \{f,g\}}$ (see Lemma \ref{max}). In particular, $f_{\alpha} \in H^{\uparrow}.$ 
Notice also that 
whenever we have such $\alpha,\beta\in 2^\omega$ that $\alpha_n>\beta_n$ for infinitely many $n\in\omega$, then  $\limsup_{n\to\infty}\frac{f_\alpha(n)}{f_\beta(n)}=\infty$ and when $\alpha_n<\beta_n$ for infinitely many $n\in\omega$, then $\limsup_{n\to\infty}\frac{f_\beta(n)}{f_\alpha(n)}=\infty$. 
It is known that there is such a family $\A \subset 2^\omega$ that $\card(\A)=\ce$ and for all $\alpha,\beta\in \A$ we have $\card(\{n\in\omega: \alpha_n=\beta_n \})<\omega$ (e.g., characteristic functions of elements of any maximal almost disjoint family satisfy those conditions). Thus, $\card(S_f)=\ce$.
\end{proof}
Now we know that there are only two possibilities: $\card( S_f )= 1$ or $\card (S_f) = \ce$. The next Theorem provides us with a condition equivalent to the first possibility. Recall that the expression "for almost all $n\in\omega$" means "for all but finitely many $n\in\omega$".
\begin{thm}
\label{moce}
For any $f \in H^\uparrow$, the condition $\card(S_f)=1$ holds if and only if there are such $M>0$ and $\varepsilon>0$ that for almost all $n\in\omega$ we have: 
$$
\frac{f(n+\lfloor\varepsilon f(n)\rfloor)}{f(n)}\leq M.
$$
\end{thm}
\begin{proof}
($\Rightarrow$):
Suppose that for every $M>0$ and $\varepsilon>0$ there are infinitely many such $n\in\omega$ that $f(n+\lfloor\varepsilon f(n)\rfloor)/f(n)>M$. We will show that $\card(S_f)\geq 2$. 

We find such a sequence $(n_k)_{k\in\omega}$ that $f(n_k+\lfloor f(n_k)/2^k\rfloor)/f(n_k)>k$, $n_{k+1} > n_k + f(n_k )/2^k$. Then $\lfloor  \frac{f(n_k)}{2^k} \rfloor \leq \frac{f(n_k)}{2^k}<n_{k+1}-n_k$ and consequently $\frac{f(n_{k+1})}{f(n_k)}=\frac{f(n_k+n_{k+1}-n_k)}{f(n_k)} \geq f(n_k+\lfloor f(n_k)/2^k\rfloor)/f(n_k)>k \geq 2$. Denote as $I_k$ the intervals $[n_k,n_k+\lfloor f(n_k)/2^k\rfloor]$. Then we define a function $g\in H^\uparrow$ given by:
$$g(n)=\left\{\begin{array}{ll}
f(n_k+\lfloor f(n_k)/2^k\rfloor) & \textrm{if } n\in I_k \textrm{ for some } k\in\omega,\\
f(n) & \textrm{otherwise}.
\end{array}\right.$$
We can easily see that $\limsup_{n\to\infty} \frac{g(n)}{f(n)}=\infty$ since $\frac{g(n_k)}{f(n_k)}>k$ for all $k\in\omega$. We can also notice that $\Z_f\subseteq\Z_g$, because $g(n)\geq f(n)$ for all natural $n$. 

To prove that $\Z_g\subseteq\Z_f$, take any $A\in \Z_g$. We will show that $A\in \Z_f$.  Clearly, $\card(A\cap n)/g(n)=\card(A\cap n)/f(n)$ when $n$ is not in any $I_k$. Thus, by the monotonicity of $f$, we have $A\setminus \bigcup_{k\in\omega}I_k\in \Z_f$. We may also observe that $\bigcup_{k\in\omega}I_k\in \Z_f$. Indeed, take any $n \in \omega$ and such $k_0 \in \omega$ that $n \in [n_{k_0},  n_{k_0 +1})$. Then: 
$$
\frac{\card(\bigcup_{k \in \omega} I_k \cap n)}{f(n)} \leq \frac{\sum_{k \leq k_0} \card( I_k )}{f(n_{k_0} )} \leq \frac{k_0 \card( I_{k_0}) }{f(n_{k_0})}  
$$
$$\leq (\lfloor \frac{f(n_{k_0})}{2^{k_0}} \rfloor +1 )\frac{k_0}{f(n_{k_0})} < 2\frac{f(n_{k_0})}{2^{k_0}} \frac{k_0}{f(n_{k_0})} = \frac{k_0}{2^{k_0-1}} \xrightarrow{k_0 \to \infty} 0 ,
$$
as $1 <  \lfloor \frac{f(n_{k_0})}{2^{k_0}} \rfloor$ by $f(n_k+\lfloor f(n_k)/2^k\rfloor)/f(n_k)>k$ (otherwise $f(n_k+\lfloor f(n_k)/2^k\rfloor)/f(n_k)=1$) and $\card(I_{k+1}) \geq \card(I_k)$ by $f(n_k ) /2^k < f(n_{k+1})/2^{k+1}$. Therefore, $A\subseteq \left(A\setminus \bigcup_{k\in\omega}I_k\right) \cup \left(\bigcup_{k\in\omega}I_k \right) \in \Z_f$, hence $A\in\Z_f$.

($\Leftarrow$):  Suppose to the contrary that $f$ satisfies our condition for some $M, \varepsilon$ and there exists such a function $g$ not $R$-equivalent to $f$ that $g\in S_f$. Then there are two possibilities: either $\limsup_{n\to\infty}g(n)/f(n)=\infty$ or $\limsup_{n\to\infty}f(n)/g(n)=\infty$. 

Assume first that for every $m>0$ there are infinitely many such $n\in\omega$ that $\frac{g(n)}{f(n)}>m$. Then we can find such a sequence $(n_k)_{k\in\omega}$ of natural numbers that $g(n_k)/f(n_k)>k$ and $n_{k+1}\geq \varepsilon f(n_k)+n_k$ for all $k\in \omega$. Next, we construct a set $A$ in such a way that $\card \left( A\cap (n_k+\lfloor\varepsilon f(n_k)\rfloor) \right) =\lfloor \varepsilon f(n_{k})\rfloor$ for each $k\in\omega$ and $A \subset \bigcup_{k \in \omega} [n_k , \lfloor \varepsilon f(n_k ) \rfloor + n_k]$. Then, $A\not\in \Z_f$, because
$$
\frac{\card\left(A\cap (n_k + \lfloor \varepsilon f(n_k) \rfloor )\right)}{f(n_k + \lfloor \varepsilon f(n_k) \rfloor)} = \frac{\lfloor \varepsilon f(n_k) \rfloor }{f(n_k + \lfloor \varepsilon f(n_k) \rfloor)} \geq \frac{ \varepsilon f(n_k)  - 1}{f(n_k + \lfloor \varepsilon f(n_k) \rfloor)} > \frac{\varepsilon }{2 M} >0
$$
for big enough $k$ as $f(n)/(f(n + \lfloor \varepsilon f(n) \rfloor)) \geq 1/M$ for almost all $n \in \omega$ and $1/(f(n_k + \lfloor \varepsilon f(n_k) \rfloor)) \to 0$. On the other hand, for each $n \in \omega$ we may fix such $k$ that $n \in [n_k , n_{k+1})$, and then:
$$
\frac{\card(A \cap n)}{g(n)} \leq 
\frac{\card(A \cap n_{k+1})}{g(n_k)} \leq
\frac{\card(A\cap (n_k + \lfloor \varepsilon f(n_k) \rfloor ) +1}{g(n_k)} = 
$$
$$
=\frac{\lfloor \varepsilon f(n_k) \rfloor +1  }{g(n_k)} \leq 
\frac{\varepsilon f(n_k )}{g(n_k)} + \frac{1}{g(n_k)} \leq
\frac{\varepsilon}{k} + \frac{1}{g(n_k)} \to 0,
$$
thus $A \in \Z_g$. Therefore, $\Z_f \neq \Z_g$ -- a contradiction. \\
Assume now that for every $m>0$ there are infinitely many such natural $n$ that $\frac{f(n)}{g(n)}>m$. This case is more complicated. We start by observing that, without loss of generality, we may assume that $f(n) > \max\{ (\frac{\varepsilon}{M} - \frac{\varepsilon}{\lfloor M \rfloor +1})^{-1} , 1/\varepsilon \}$ for all natural $n$ (by replacing $f(n)$ with $\max\{ 1 + \lceil (\frac{\varepsilon}{M} - \frac{\varepsilon}{\lfloor M \rfloor +1})^{-1} \rceil , \lceil 1+ 1/\varepsilon \rceil, f( n ) \}$ for each $n \in \omega$). \\
Now we will construct two sequences: $(n_k)_{k\in\omega}$ and $(m_k)_{k\in\omega}$. Let $m_{0}=n_0=0$. If $m_i$ are defined for all $i\leq k$, then we pick such $m_{k+1}$ that $m_{k} + 1 + \varepsilon f(m_{k}+1)<m_{k+1}$ and $f(m_{k+1})/g(m_{k+1})>k+1$ (observe that $m_{k+1}>m_k$). As we have defined the sequence $(m_k)$, we move to the other one. For each $k \leq M$ we set $n_k = 0$, while for $k>M$ there is such $n_k$ that: 
$$n_k+\frac{\varepsilon f(n_k)}{k}<m_k<n_k+\varepsilon f(n_k).$$
Indeed, if this is not the case, then there must be such $n > m_{k-1}$ that: 
$$
n+\varepsilon f(n)\leq m_k\leq (n+1)+\frac{\varepsilon f(n+1)}{k},
$$
hence we obtain 
\begin{equation}\label{eqcard1}
f(n)\varepsilon\leq 1+\frac{\varepsilon f(n+1)}{k}.
\end{equation}
On the other hand, by $f(n)>1/\varepsilon$ and monotonicity of $f$, we get 
\begin{equation}\label{eqcard2}
\frac{f(n+1)}{f(n)}\leq \frac{f(n+\lfloor\varepsilon f(n)\rfloor)}{f(n)}\leq M \mbox{, thus } \frac{f(n+1)}{M} \leq f(n) . 
\end{equation} 
By \eqref{eqcard1} and \eqref{eqcard2}, we get:
$$ 
\frac{\varepsilon f(n+1)}{M}\leq 1+\frac{\varepsilon f(n+1)}{k},
$$ 
and hence $f(n+1)<(\frac{\varepsilon}{M}-\frac{\varepsilon}{k})^{-1} \leq ( \frac{\varepsilon}{M } - \frac{\varepsilon}{\lfloor M \rfloor +1} )^{-1} $, which contradicts $f(\cdot) > (\frac{\varepsilon}{M} - \frac{\varepsilon}{\lfloor M \rfloor +1})^{-1}$. \\
Note that since $(m_k + 1) + \varepsilon f (m_k + 1) <m_{k+1} < n_{k+1} + \varepsilon f(n_{k+1})$, we get that $m_k < n_{k+1}$. On the other hand, clearly $n_k < m_k$, so for each $k>M$ we have $n_k < m_k < n_{k+1}$. \\
We construct a set $A$ in such a way that $\card(A\cap m_k)=\max_{t \leq k} \lfloor \varepsilon f(n_{t})/t\rfloor$ and $A\cap n_{k+1}\setminus m_k=\emptyset$ for all $k\in\omega$ (this is clearly possible by an easy inductive construction). We may also notice that if $\max_{t \leq k} \lfloor \varepsilon f(n_{t})/t\rfloor\not=\lfloor \varepsilon f(n_{k})/k\rfloor$, then $A\cap [n_k,m_k]=\emptyset$. We also need to observe that $f(m_k)/f(n_k)\leq f(n_k+\lfloor\varepsilon f(n_k)\rfloor)/f(n_k)<M$. Then, since $\frac{f(m_k)}{g(m_k)}>k$, we have:
$$
\frac{\card(A\cap m_k)}{g(m_k)} = 
\frac{\max_{t \leq k} \lfloor \varepsilon f(n_{t})/t\rfloor}{g(m_k )} \geq 
\frac{ \lfloor \varepsilon f(n_k)/k\rfloor}{g(m_k )} \geq 
\frac{ \varepsilon f(n_k) -1}{ k g(m_k )} \geq
\frac{ \varepsilon f(n_k) -1}{ k (f(m_k)/k)} \geq
\frac{\varepsilon f(n_{k})-1}{ f(m_k)}
>\frac{\varepsilon}{2 M}
$$
for big enough $k$, as $\frac{1}{f(m_k)}\to 0$. Thus, $A \notin \Z_g$, hence $A$ is infinite and $f(n_k)/k\to\infty$ when $k\to\infty$. We will use this to prove that $\card(A \cap n)/f(n) \to 0$.
Take $n\in\omega$ and let $k\in\omega$ be the biggest natural number such that $n\geq n_k$ and $A\cap [n_k,m_k]\not=\emptyset$ (thus, $\card(A \cap m_k)= f(n_k)/k$). Then:
$$\frac{\card(A\cap n)}{f(n)}\leq 
\frac{\card (A \cap m_k)}{f(n_k)}=
\frac{ \lfloor \varepsilon f(n_{k})/k\rfloor}{f(n_k)}\leq 
\frac{\varepsilon }{k},
$$

hence $A \in \Z_f$. We proved that $\Z_g \neq \Z_f$ -- a contradiction. 
\end{proof}

\begin{Rem}
Observe that the following conditions are equivalent:
\begin{itemize}
\item[(i)] the condition from Proposition \ref{moce};
\item[(ii)] there is such $M>0$ that for almost all $n\in\omega$ we have $\frac{f(n+\lfloor f(n)/M\rfloor)}{f(n)}\leq M$;
\item[(iii)] there are such $M>0$ and $\varepsilon >0$ that for all $n\in\omega$ we have $\frac{f(n+\lfloor \varepsilon f(n)\rfloor)}{f(n)}\leq M$;
\item[(iv)]there is such $M>0$ that for all $n\in\omega$ we have $\frac{f(n+\lfloor f(n)/M\rfloor)}{f(n)}\leq M$.
\end{itemize} 
\end{Rem}
Now we provide two much easier conditions which are sufficient to calculate $\card(S_f)$ in some cases. 
\begin{prop}\label{tryw1}
Let $f\in H^\uparrow$.
\begin{itemize}
\item If there are such an increasing sequence $(n_k)_{k\in\omega}$ and $l \in \omega$ that $\lim_{k \to \infty} f(n_k + l)/f(n_k) = \infty$, then $\card(S_f) = \mathfrak{c}$.
\item If $\limsup_{n\to\infty} (f(n+1) - f(n)) < \infty$, then $\card(S_f)=1$.
\end{itemize}  
\end{prop}
\begin{proof}
We start with proving the first part of the Proposition. In view of Proposition \ref{moce2} and Theorem \ref{moce} it is enough to prove that for each $\varepsilon, M >0$ there exist infinitely many such $n$ that $\frac{f(n+\lfloor\varepsilon f(n)\rfloor)}{f(n)}> M$, so we fix $\varepsilon, M>0$ and such $k_0 \in \omega$ that $\varepsilon f(n_k)>l$ and $f(n_k +l)/f(n_k) > M$ for each $k> k_0$. Thus, for every such $k$ we have $\frac{f(n_k+\lfloor\varepsilon f(n_k)\rfloor)}{f(n_k)}\geq \frac{f(n_k+\lfloor l\rfloor)}{f(n_k)} = \frac{f(n_k+l)}{f(n_k)} > M$.\\
Now we move to proving the second statement. Again, we will base on Proposition \ref{moce2} and Theorem \ref{moce}. Let us denote $\alpha := \limsup_{n\to\infty} (f(n+1) - f(n))$. Then, for big enough $n$ we have:
$$
\frac{f(n+ \lfloor  f(n)\rfloor)}{f(n)} = \frac{f(n+ \lfloor  f(n)\rfloor) - f(n+ \lfloor  f(n)\rfloor -1) + \ldots +f(n+ 1) - f(n) + f(n)}{f(n)} \leq 
$$
$$
\leq \frac{\alpha \lfloor  f(n) \rfloor + f(n)}{f(n)} \leq \frac{(\alpha+1)  f(n)}{f(n)}=\alpha+1,
$$
and we conclude our claim. 
\end{proof}
Note that if the first case of the previous Proposition holds, then one can produce an increasing sequence $(m_k)$ satisfying $\lim_{k \to \infty} \frac{f(m_k)+1}{f(m_k)}=0.$

Proposition \ref{tryw1} allows us to easily produce plenty of examples of functions $f \in H$ for both cases: $\card(S_f) = 1$ and $\card(S_f) = \mathfrak{c}$. However, as one may suppose, those conditions do not provide us with a characterization, as shown in the following examples. 
\begin{ex}
Let us define $f: \omega \to \omega$ by: 
$$f(n)=\left\{\begin{array}{ll}
2^k & \textrm{if } n \in [2^k , 2^k +k)\textrm{ for }k\in\omega,\\
n & \textrm{otherwise.}
\end{array}\right.$$
Then clearly $f$ is nondecreasing and $f \in H$. Moreover, $f(n) \in [n - \log_2(n) , n]$, hence $\frac{f(n)}{n} \to 1$ and consequently $\Z_f = \Z$, thus $\card(S_f) =1 $, as the identity function satisfies the condition from Proposition \ref{moce}. On the other hand, clearly $\limsup_{n\to\infty} (f(n+1) - f(n)) = \infty$. 
\end{ex}

\begin{ex}
Let us fix such an increasing sequence $(m_k)$ that $m_0 = 1$, $m_1 = 1$ and $m_{k+1} > \max\{ k m_k ,m_k +k\} $. Then we set $f\in H$ as:
$$f(n)=\left\{\begin{array}{ll}
0 & \textrm{if } n=0, \\
1 & \textrm{if } n=1, \\
l f(m_k) & \textrm{if } n = m_k + l \textrm{ for }k\geq 1 \textrm{ and } l \in \{ 1, \ldots k\},\\
f(m_k + k) & \textrm{if } n \in (m_k + k, m_{k+1}] \textrm{ for } k\geq 1.
\end{array}\right.$$
We also set $g \in H$ as:
$$g(n)=\left\{\begin{array}{ll}
0 & \textrm{if } n=0, \\
1 & \textrm{if } n=1, \\
f(m_{k+1}) & \textrm{if } n \in (m_k , m_{k+1}].
\end{array}\right.$$
Since $m_{k+1} > m_k + k$, the above definitions are correct, and since $m_{k+1}> k m_k $, one may easily check in the inductive way that $\on{id}_\omega \geq g$. Note that $g(m_{k}+1)/g(m_k) = k$, so, by Proposition \ref{tryw1}, $\card(S_g) = \mathfrak{c}$. On the other hand, for any $l \in \omega$ and $n>0$ we get $ f(n+l)/f(n) \leq l$. It remains to prove that $\Z_g = \Z_f$. Since $g \geq f$, it suffices to prove that $\Z_g \subset \Z_f$, so let us fix $A \in \Z_g$ and set $A' = A \setminus \bigcup_{k \in \omega}(m_k , m_k +k]$. Then $A' \in \Z_f$ since $\card(A'\cap n)/f(n) = \card(A'\cap n)/g(n)$ for $n \in \bigcup_{k \in \omega} (m_k + k , m_{k+1}]$ and $\card(A'\cap n)/f(n) \leq \card(A'\cap m_k )/f(m_k) = \card(A'\cap m_k )/g(m_k)$ if $n \in (m_k , m_k +k]$. Thus, it only remains to prove that $B:= \bigcup_{k \in \omega}(m_k , m_k +k] \in \Z_f$. But this is trivial since  one may easily check that $f(m_k)=(k-1)!$ for $k \geq 1$, so if $n = m_k +l$ for some $k \in \omega, l \in \{1, \ldots , k\}$, then 
$$
\frac{\card(B \cap n)}{f(n)} = \frac{1 + 2 + \ldots + (k-1) +l-1}{(k-1)!\cdot l}\leq \frac{1 + 2 + \ldots + (k-1) +k}{(k-1)!} = \frac{(k+1)k/2}{(k-1)!} \to 0,
$$
and if $n \in(m_{k} +k, m_{k+1}]$, then $\card(B \cap n)/f(n)=\card(B \cap (m_k+k))/f(m_k +k) \to 0$ by the previous case. 
\end{ex}

\section{Preliminaries on ordering}\label{sekwstepmnog}
This Section concentrates on various, rather well-known, results which will be useful in Section $5$ of our paper. \\
We start with the notion of bounding number. For more details see e.g. \cite{Buk}.
\begin{df}
Let $(X, \preceq )$  be a partially ordered set (poset in short). We say that $A \subset X$ is \emph{bounded} if there exists such $a_0 \in X$ that for each $a \in A$  we have $a \preceq a_0$, and we say that a set is \emph{unbounded} if it is not bounded. By \emph{bounding number} of such poset we denote the following cardinal: 
$$\mf{b} (X, \preceq ) = \min\{ \card(A) : A \subset X \mbox{ is unbounded} \}.$$  
Moreover, we define one special bounding number, namely $\mf{b} = \mf{b} ( \omega^{\omega}, \leq^{\star})$, where for $\alpha , \beta \in \omega^{\omega}$ we set $\alpha \leq^{\star} \beta $ if $\alpha(n) \leq \beta (n)$ for all but finitely many $n \in \omega$.
\end{df}
Now we will prove two rather folklore lemmas.
\begin{lem}\label{lemat1}
Let $(X, \preceq)$ be a poset and $Y \subset X$. If there is a function $\varphi \colon X \to Y$ with $\varphi(x) \succeq x$, then $\mf{b}(X,\preceq)=\mf{b}(Y,\preceq),$ where the order in $Y$ is inherited from $X$.
\end{lem}
\begin{proof}
First we will prove that any unbounded set $B$ in $X$ determines an unbounded set $C$ in $Y$ with $\card(C) \leq \card(B)$. Obviously, for any $A \subset X$ the inequality $\card(\varphi[A]) \leq \card(A)$ holds. Assume now that $A$ is unbounded in $X$. Supposition that $\varphi[A]$ is bounded in $Y$ by some $y \in Y$ provides inequalities $x \preceq \varphi(x) \preceq y,$ satisfied for each $x \in A$. This contradicts the unboundedness of $A$ in $X$, because $y \in Y \subset X$. Thus, we infer that $\mf{b}(Y,\preceq) \leq \mf{b}(X,\preceq)$.  For the second inequality, consider an unbounded subset $A$ of $Y$ and suppose that $A$ is bounded in $X$ by some $x \in X$. Then, for all $y \in A$ we have $y \preceq x \preceq \varphi(x) \in Y,$ which contradicts the unboundedness of $A$ in $Y$. This proves that $\mf{b}(X,\preceq) \leq \mf{b}(Y,\preceq).$
\end{proof}
\begin{cor}
$\mf{b} = \mf{b} (\omega^{\omega} \uparrow , \leq^{\star} )$, where $\omega^{\omega} \uparrow $ stands for the set $\{\alpha \in \omega^\omega : \alpha \mbox{ is strictly increasing} \}$.
\end{cor}
\begin{proof}
By the above Lemma, it suffices to show that for any $\alpha \in \omega^{\omega}$ there exists $\beta \in \omega^{\omega} \uparrow$ satisfying $\alpha \ls \beta$. Indeed, take any $\alpha \in \omega^{\omega}$ and set $\beta (0) = \alpha(0) + 1$, $\beta(n+1) = \max \{ \alpha(n+1), \beta(n) \} + 1$. Then, obviously, $\alpha(n) \leq \beta (n)$ for each $n \in \omega$ and $\beta$ is strictly increasing.
\end{proof}
\begin{lem}
$\mf{b}=\mf{b}(H,\geq^{\star})$.
\end{lem}
\begin{proof}
Define $A=\{\beta \in \omega^{\omega} \colon \beta \textrm{ is a nondecreasing surjection} \}$. We will prove that $\mf{b}=\mf{b}(A,\geq^{\star})=\mf{b}(H,\geq^{\star})$. Thanks to the previous Corollary, we know that $\mf{b} = \mf{b} (\omega^{\omega} \uparrow , \leq^{\star} ).$ Let us define a function $f \colon \omega^{\omega} \uparrow \to A$ in the following way$\colon$
$$ \left(f(\alpha)\right)(k)= \on{min} \{ n \in \omega : k \leq \alpha(n) \}.$$
Since each $\alpha \in \omega^\omega \uparrow$ is increasing, we get that $f(\alpha)$ is nondecreasing and surjective. This proves that, indeed, $f(\alpha) \in A$.
Now we will show that $f$ is a one-to-one correspondence. Take any distinct $\alpha_1 ,\alpha_2 \in \omega^{\omega} \uparrow$ and pick the smallest $n_0 \in \omega$ with $\alpha_1(n_0) \neq \alpha(n_0)$. 
Without loss of generality, we may assume that $\alpha_1 (n_0) < \alpha_2 (n_0)$. Then for any $m \in (\alpha_1 (n_0) , \alpha_2 (n_0)]$ we get $f(\alpha_1 ) (m)\geq n_0 +1 >f(\alpha_2)(m)$, thus $f$ is injective. To prove that $f$ is surjective, we simply claim that for $\beta \in A$ we have $\beta=f(\alpha)$, where $\alpha(n) = \on{max}\{ k \in \omega : \beta(k)=n \}$. 
Now we will show the equivalence 
$$\forall_{\alpha_1,\alpha_2 \in \omega^{\omega} \uparrow}  \ \alpha_1 \leq^{\star} \alpha_2 \Leftrightarrow f(\alpha_1) \geq^{\star} f(\alpha_2).$$
Take any $\alpha_1,\alpha_2 \in \omega^{\omega} \uparrow$ with $\alpha_1 \leq^{\star} \alpha_2$. Pick such $n_0 \in \omega$ that for all $n\geq n_0$ we have $\alpha_1(n) \leq \alpha_2(n)$ and take any $k>\alpha_1(n_0)$. Then 
$$
 (f(\alpha_1))(k) =\on{min}\{n \in \omega : k \leq \alpha_1(n) \} \geq \on{min}\{n \in \omega : k \leq \alpha_2(n) \} = (f(\alpha_2))(k)  ,
$$
as $\{n \in \omega : k \leq \alpha_1(n) \} \subset \{ n \in \omega: \alpha_1(n) \leq \alpha_2(n)\}$. Since $f$ is a bijection, we obtained also the second implication.
We proved that $(A,\geq^{\star})$ and $(\omega^{\omega} \uparrow, \leq^{\star})$ are isomorphic, so $\mf{b}=\mf{b}(A,\geq^{\star})$. \\
Now, let us observe that $A \subset H$. Indeed, since each $\beta
\in A$ is a nondecreasing surjection, we obtain that $\beta(n) \to \infty$ and $\beta(n) \leq n$, thus $\beta \in H$. \\
Now we will define such function $h \colon H \to A$ that $h(\gamma) \leq^{\star} \gamma$ for all $\gamma \in H$.
Take $\gamma \in H$. We will proceed inductively. Put $h(\gamma)(0)=0$ and $h(\gamma)(n+1) = \min\{ h(\gamma)(n)+1, \gamma(n+1) , \gamma(n+2), \gamma(n+3), \ldots \}$. Clearly, $h(\gamma)$ is nondecreasing and $h(\gamma)(n) \leq \gamma(n)$ for all $n \in \omega$. Since $h(\gamma)(n+1) - h(\gamma)(n) \leq 1$, it is enough to use the fact that $\gamma(n) \to \infty$ to obtain that $h(\gamma)$ is surjection. By Lemma \ref{lemat1}, we conclude our claim.
\end{proof}

\section{Ordered set $\{\Z_g : g \in G \}$}\label{sekporzadekmiedzyidealami}
There are many known partial preorders (i.e., relations which are reflexive and transitive) in the family of all ideals on $\omega$. In this Section we present some results showing how simple density ideals behave with respect to two of them: inclusion, which is one of the strongest, and Kat\v{e}tov order, which is one of the weakest. As a corollary we obtain $\mathfrak{c}$ many non-isomorphic simple density ideals. We start with strengthening of \cite[Theorem 2.7]{Den} where the authors have constructed an antichain of cardinality $\mathfrak{c}$ with respect to inclusion.
\begin{thm}\label{antyincl}
For any $A \subset G$ satisfying $\card(A)< \mathfrak{b}$ there exists such a family $G_0 \subset H^\uparrow$ of cardinality $\mathfrak{c}$ that for any $f \in A$, $f_1, f_2 \in G_0$ the ideals $\mathcal{Z}_{f}, \mathcal{Z}_{f_1}, \mathcal{Z}_{f_2}$ are pairwise incomparable with respect to inclusion whenever $f_1 \neq f_2$.
\end{thm}

\begin{proof}
Take any $A \subset G$ with $\card(A)<\mathfrak{b}$. Without loss of generality we may assume that $A \subset H^\uparrow$. Since $\mathfrak{b}=\mathfrak{b}(H^\uparrow, \geq^{\star})$, there are such $g_1 \in H^\uparrow, \ g_2 \in \omega^\omega$ that $g_1 \leq^{\star} f \leq^{\star} g_2$ for all $f \in A$. By some finite modifications, we may additionally assume that $g_1 \leq g_2$. Fix a maximal almost disjoint family $D$ of subsets of $\omega$ with $\card{D}=\mathfrak{c}$ (for the existence of such a family see e.g. \cite[Theorem 5.35]{Buk}).
\newline
Since $g_2 \in G$, $\frac{n}{g_2(n)} \nrightarrow 0$. Then we may fix such $a \in (0,\frac{1}{3})$ that $\frac{n}{g_2(n)}>a$ for infinitely many $n$. 
\newline
Now we will define sequences $(a_n), (l_n), (k_n)$ and values $h(a_n)$. We start with setting $a_0=0$. Assume that $a_n$ is defined. Choose $k_n$ satisfying $\frac{k_n}{g_2(a_n+k_n)}>a.$ Now, find such $h(a_n) \in \omega$ that
$$
\frac{a_n+k_n}{h(a_n)}<\frac{1}{n} \mbox{ and } h(a_n)\geq g_2(a_n).
$$
Next we fix $l_n>a_n+ k_n $ with 
\begin{equation}
\frac{l_n - 1}{h(a_n)}>a. \label{eq:l} \end{equation} 
Further, we look for such $a_{n+1}>l_n + a_n$ that the inequalities
$$
\frac{a_n+l_n}{g_1(a_{n+1}-l_n)}<\frac{1}{n} \mbox{ and } g_1(a_{n+1}) \geq h(a_n) + 1
$$
hold.
\newline
Note that the sequence $(a_n)$ is increasing, hence we may define $I_n:=[a_n,a_{n+1})$ and $h(m)=h(a_n)$ whenever $m \in I_n$. Since $\bigcup_{n \in \omega} I_n = \omega$, we obtain that $h \in \omega^{\omega}$ and $h(n) \rightarrow \infty$.
\newline
For $P \in D$ set 
$$
h_P(n)=\begin{cases} h(n) &\mbox{ if } n \in \bigcup_{m \in P} I_m, \\ \lfloor\frac{g_1(n)+g_2(n)}{2}\rfloor &\mbox{ otherwise.} \end{cases}
$$
Set $G_0:=\{h_P \colon P \in D\}.$ We will show that $G_0 \subset H^\uparrow$. Fix $P \in D$.
\newline
Obviously, since $h(n) \rightarrow \infty$ and $g_1,g_2 \in G$, $h_P(n) \rightarrow \infty$. Recall that $a_n<l_n<a_{n+1}$ for each $n \in \omega$ and consequently, $l_n \in I_n$ for $n \in \omega$. The definitions of $h_P$ and $h$ together with the condition (\ref{eq:l}) guarantee that
$$\frac{l_n}{h_P(l_n)}=\frac{l_n}{h(l_n)}=\frac{l_n}{h(a_n)}>a$$
whenever $n \in P$. Since $P$ is infinite, we may consider the sequence $(l_n)_{n \in P}$ and thus 
$$\limsup_{n \to \infty} \frac{n}{h_P(n)} \geq \limsup_{n \to \infty} \frac{l_n}{h_P(l_n)} \geq \limsup_{n \to \infty,n \in P} \frac{l_n}{h_P(l_n)}\geq a,$$
hence $h_P \in G$. By $h(a_n) \in \omega$ we get $h_P \in H$ and from $g_1 (a_n) \leq g_2(a_n)) \leq h(a_n) \leq  g_1(a_{n+1}) - 1 \leq g_2 (a_{n+1}) \leq h(a_{n+1})$ the monotonicity follows. 
\newline
The next step is to show that $\mathcal{Z}_{h_{P}},\mathcal{Z}_{h_U},\mathcal{Z}_f$ are incomparable whenever $P,U \in D, f \in A$. Let us fix an increasing enumeration $\{h_n : n \in \omega \} = P \setminus U$. We start with finding such $B \in \mathcal{Z}_{h_P}$ that 
$$
\frac{\card(B \cap (a_{h_n} + k_{h_n}))}{g_2(a_{h_n} + k_{h_n})} \not\to 0. 
$$
Once it is done, it will give us $\Z_{h_P} \not\subset \Z_f, \Z_{h_U}$ since $f(a_{h_n} + k_{h_n}),h_U (a_{h_n} + k_{h_n}) \leq g_2(a_{h_n} + k_{h_n})$. \\
We set $B:= \bigcup_{n \in \omega} [a_{h_n}, a_{h_n} + k_{h_n})$. Then 
$$
\frac{\card(B \cap (a_{h_n}+k_{h_n}))}{g_2(a_{h_n} + k_{h_n})} \geq
 \frac{k_{h_n}}{g_2(a_{h_n} + k_{h_n})} >a>0.
$$
On the other hand, for any $m \in \omega$ we may find such $n \in \omega$ that $m \in [a_{h_n} , a_{h_{n+1}})$, and we get
$$
\frac{\card(B \cap m)}{h_P(m)} \leq
\frac{\card(B \cap a_{h_{n+1}})}{h_P(a_{h_n})} \leq
\frac{a_{h_n} + k_{h_n}}{h_P(a_{h_n})} \leq \frac{1}{h_n} \to 0.
$$
We proved that $\Z_{h_P} \not\subset \Z_{h_U},\Z_f$. In fact, we also proved that $\Z_{h_U} \not\subset \Z_{h_P}$. However, the proof of the Theorem is not complete -- we still must show that $\Z_f \not\subset \Z_{h_P}$. Let us set $C:= \bigcup_{n \in \omega} (a_{h_n+1}-l_{h_n},a_{h_n+1}]$. Then
$$
\frac{\card(C \cap (a_{h_n +1}-1))}{h_P(a_{h_n +1}-1)} \geq
\frac{l_{h_n}-1}{h_P(a_{h_n +1}-1)} 
= \frac{l_{h_n}-1}{h_P(a_{h_n})}>a,
$$
so $C \notin \Z_{h_P}$. On the other hand, for any $m \geq a_{h_1}$ we may find such $n \in \omega$ that $m \in [a_{h_n +1}- l_{h_n} , a_{h_{n+1}+1} - l_{h_{n+1}})$. The inequality $f(m) \geq g_1 (m)$ is satisfied from some point on, thus for big enough $m$ we get
$$
\frac{\card(C \cap m)}{f(m)} \leq
\frac{\card(C \cap m)}{g_1 (m)} \leq
\frac{\card(C \cap (a_{h_{n+1}+1} - l_{h_{n+1}}))}{g_1(a_{h_n +1}- l_{h_n})} \leq
\frac{a_{h_n} + l_{h_n}}{g_1(a_{h_n +1}- l_{h_n})} \leq  \frac{1}{h_n} \to 0,
$$
so $C \in \Z_f$, thus $\Z_f \not\subset \Z_{h_P}$ and we are done.
\end{proof}
Now, we move to a more general preorder among ideals on $\omega$. Namely, we say that an ideal $\I$ is below an ideal $\J$ in the Kat\v{e}tov order ($\mathcal{I}\leq_{K}\mathcal{J}$) if there is such a function (not necessarily a bijection) $\phi\colon\omega\to\omega$ that $A\in\I$ implies $\phi^{-1}[A]\in\J$ for all $A\subseteq\omega$. A property of ideals can often be expressed by finding a critical ideal in the sense of $\leq_K$ with respect to this property (see e.g. \cite{WR}). This approach is very effective for instance in the context of ideal convergence of sequences of functions (see e.g. \cite{zReclawem} and \cite{zMarcinem}). Therefore, the structure of $\{\Z_g:g\in G\}$ ordered by $\leq_K$ can give us some information about the properties of simple density ideals.

\begin{thm}\label{antikate}
Among ideals of the form $\Z_g$, for $g\in H$, there exists an antichain in the sense of $\leq_{K}$ of size $\ce$.
\end{thm}
\begin{proof}
We will use a slight modification of the family of size $\ce$ used in the proof of \cite[Theorem 2.7]{Den}. Fix a family $\mathcal{F}$ of cardinality $\ce$ of infinite pairwise almost disjoint subsets of $\omega$ and such $\alpha:\omega \to \omega$ that $2\alpha(n+1) -1 > (2\alpha(n)+1)!$. Function $\alpha$ is a kind of technical complication introduced due to Case 2 in the sequel -- if one wants to understand the idea of the proof, it might be worthy to firstly read it with replacing $\alpha(i)$ by $i$ and noting where this case crushes (but the rest of our proof still works), and then read the complete proof again. For each $M\in\mathcal{F}$ let $(m_i)_{i\in\omega}$ be its increasing enumeration and define
$$f_M(n)=\left\{\begin{array}{ll}
(2\alpha(m_i))! & \textrm{if } (2\alpha(m_i)-1)!< n\leq (2\alpha(m_i)+1)!\textrm{ for }i\in\omega,\\
n & \textrm{otherwise.}
\end{array}\right.$$
It is easy to see that $f_M\in H$. 

Fix $K,M\in\mathcal{F}$. We will show that $\Z_{f_M}\not\leq_K\Z_{f_K}$. Suppose to the contrary that there is such $\phi\colon\omega\to\omega$ that $\phi^{-1}[A]\in\Z_{f_K}$ for all $A\in\Z_{f_M}$. 

Let $(k_i)_{i\in\omega}$ be an increasing enumeration of $K\setminus (M\cup\{0\})$. For each $i\in \omega$ denote $I_{i}=[(2\alpha(i)-1)!,(2\alpha(i)+1)!)$ and
$$B_{i}=\{n\in I_{i}:\ \phi(n)\geq (2\alpha(i)+1)!\},$$
$$C_{i}=\{n\in I_{i}:\ 2\alpha(i)+1\leq \phi(n)<(2\alpha(i)+1)!\},$$
$$D_{i}=\{n\in I_{i}:\ \phi(n)<2\alpha(i)+1\}.$$

Observe that $(2\alpha(i))!+\alpha(i)(2\alpha(i))!+\alpha(i)(2\alpha(i))!=(2\alpha(i)+1)!\leq 2\card(I_{i})$ whenever $\alpha(i)>1$. Therefore, at least one of the following three cases must happen. 

\textbf{Case 1.: }For infinitely many $i\in\omega$ we have $\card(B_{k_i})\geq (2\alpha(k_i))!/2$. Let $(t_j)_{j\in\omega}$ be an increasing enumeration of the set of $k_i$ with such property. We may additionally assume that $\max \phi[B_{t_j}]<(2\alpha(t_{j+1})+1)!$. For each $j\in\omega$ let $B'_j$ be any subset of $B_{t_j}$ of cardinality $(2\alpha(t_{j}))!/2$. Define $B=\bigcup_{j\in\omega}B'_j$. Then 
$$
\frac{\card(B \cap (2\alpha(t_j) +1)!)}{f_K ((2\alpha(t_j) +1)!)} \geq 
\frac{(2\alpha(t_{j}))!/2}{(2\alpha(t_j) )!} = \frac{1}{2},
$$
hence, $\phi^{-1}[\phi[B]]\supseteq B\notin\Z_{f_K}$. However, $\phi[B]\in\Z_{f_M}$. Indeed, for $n \in \omega$ we find such a $j \in \omega$ that $n \in [(2\alpha(t_j) +1)!,(2\alpha(t_{j+1}) +1)!)$ and then
$$
\frac{\card(\phi[B] \cap n)}{f_M (n)} \leq 
\frac{\card(\phi[B] \cap (2\alpha(t_{j+1}) +1)!)}{f_M ((2\alpha(t_j) +1)!)} \leq 
\frac{\sum_{i \leq j} (2\alpha(t_i))!/2}{(2\alpha(t_j) +1)!} \leq 
\frac{j (2\alpha(t_{j})-1)!+ (2\alpha(t_j))!}{2((2\alpha(t_j) +1)!)} \leq
$$
$$
\leq \frac{2\alpha(t_j) (2\alpha(t_{j})-1)!+ (2\alpha(t_j))!}{2((2\alpha(t_j) +1)!)}
= \frac{ 2(2\alpha(t_{j}))!}{2((2\alpha(t_j) +1)!)} = \frac{1}{2\alpha(t_j) + 1} \to 0.
$$
Thus, $B \notin \Z_{f_K}$ and $\phi(B) \in \Z_{f_M}$ -- a contradiction. 

\textbf{Case 2.: }For infinitely many $i\in\omega$ we have $\card(C_{k_i})>\alpha(k_i)(2\alpha(k_i))!/2$. Let $(t_j)_{j\in\omega}$ be an increasing enumeration of the set of $k_i$ with such property. We may additionally assume that $(j+1)^2 2(\alpha(t_j)+1)!<(2\alpha(t_{j+1})+1)$. For each $j\in\omega$ and $l \in \{ 1,2\ldots,(2\alpha(t_j))!-1 \}$ pick $e_{j,l}\in[l(2\alpha(t_j)+1),(l+1)(2\alpha(t_j)+1))$ which maximizes value $\card(\phi^{-1}[\{e_{j,l}\}]\cap I_{t_j})$. \\ 
Define $E_j=\{e_{j,l}:\ l\leq (2\alpha(t_j))!-1\}$ and $E=\bigcup_{j\in\omega}E_j$. Note that $\card(\phi^{-1}[E_j]\cap I_{t_j})\geq(2\alpha(t_j))!/8$. Indeed, otherwise we would have $\card(C_{t_j})\leq(2\alpha(t_j)+1)\card(\phi^{-1}[E_j]\cap I_{t_j})<(2\alpha(t_j)+1)(2\alpha(t_j))!/8<\alpha(t_j)2\alpha(t_j))!/2\leq\card(C_{t_j})$ (recall that $k_i>0$) -- a contradiction. Therefore, we obtain
$$
\frac{\card(\phi^{-1}[E]\cap [(2\alpha(t_j)+1)! -1])}{f_K((2\alpha(t_j)+1)!-1)} \geq
\frac{\card(\phi^{-1}[E_j]\cap [(2\alpha(t_j)+1)!-1])}{(2\alpha(t_j))!}\geq 
\frac{(2\alpha(t_j))!/8 - 1}{(2\alpha(t_j))!} \geq \frac{1}{16}
$$
for big enough $j$, so it follows that $\phi^{-1}[E]\notin\Z_{f_K}$. However, for any $n \in \omega$ we may fix such a $j \in \omega$ and such $l \in \{ 1,\ldots , (2\alpha(t_j))!-1\}$ that $n \in [l(2\alpha(t_j) +1),(l+1)(2\alpha(t_j) +1))$ or just such a $j \in \omega$ that $n \in [(2\alpha(t_j)+1)!, 2\alpha(t_{j+1})+1)$, so we consider two cases:
\begin{itemize}
 \item in the first one we get (note that $f_M = \on{id}$ on $[2\alpha(t_j)+1, (2\alpha(t_j)+1)!)$):
$$
\frac{\card(E\cap n)}{f_M(n)}\leq
\frac{\card(E\cap (l+1)(2\alpha(t_j) +1))}{f_M(l(2\alpha(t_j) +1))} =
\frac{l+\sum_{i<j}(2(\alpha(t_i))! - 1)}{l(2\alpha(t_j)+1)} \leq
$$
$$
\leq \frac{1}{2\alpha(t_j)+1}+\frac{j(2(\alpha(t_{j-1}))! - 1)}{l(2\alpha(t_j)+1)} <\frac{1}{\alpha(t_j)} +\frac{2\alpha(t_j) +1}{(j+1)l(2\alpha(t_j)+1)} \leq \frac{1}{\alpha(t_j)} + \frac{1}{j+1} \to 0;
$$
\item and in the second one:
$$
\frac{\card(E\cap n)}{f_M(n)}\leq 
\frac{\card(E\cap 2\alpha(t_{j+1})+1))}{f_M((2\alpha(t_j)+1)!)}=
\frac{\card(E\cap (2\alpha(t_{j})+1)!)}{f_M((2\alpha(t_j)+1)!)}
\leq 
$$
$$
\leq
\frac{\card(E\cap (2\alpha(t_{j})+1)!-1)+1}{f_M((2\alpha(t_j)+1)!-1)} \to 0
$$
by the previous case.
We proved that $E \in \Z_{f_M}$ and $\phi^{-1}(E) \notin \Z_{f_K}$ -- a contradiction. 
\end{itemize}

\textbf{Case 3.: }For infinitely many $i\in\omega$ we have $\card(D_{k_i})>\alpha(k_i)(2\alpha(k_i))!/2$. Let $T$ consist of all $k_i$ with such property. We inductively pick points $d_j\in\omega$, $l_{j}\in\omega$ and $t_{j}\in T$ for $j\in\omega$. We start with $d_0=l_0=t_0=0$. If all $d_j$ for $j\leq m$ are picked, then let us find such $l_{m+1}>d_{m}$ that $(m+3)/f_M(l_{m+1})<1/m$. The interval $[0,l_{m+1})$ is in $\Z_{f_M}$ as a finite set, so there must be such $t_{m+1}\in T$ that $\card(\{n\in D_{t_{m+1}}:\ \phi(n)\geq l_{m+1}\})>\alpha(t_{m+1})(2\alpha(t_{m+1}))!/4$. Indeed, otherwise $\phi^{-1}[[0,l_{m+1})]\notin\Z_{f_K}$, as 
$$
\frac{\card(\phi^{-1}[[0,l_{m+1})]\cap (2\alpha(t)+1)!-1)}{f_K (2\alpha(t)+1)!-1)} \geq 
\frac{\card(\phi^{-1}[[0,l_{m+1})]\cap I_t}{(2\alpha(t))!)} \geq
$$
$$
\geq \frac{\alpha(t)(2\alpha(t))!/2 - \alpha(t)(2\alpha(t))!/4}{(2\alpha(t))!}=
\alpha(t)/4
$$
for all $t\in T$. Now we fix such $d_{m+1}$ that $l_{m+1}\leq d_{m+1}<2\alpha(t_{m+1})+1$ and $\card(\phi^{-1}[\{d_{m+1}\}])\geq(2\alpha(t_{m+1}))!/12$. Note that this is possible, since otherwise we would have 
$$
\card(\{n\in D_{t_{m+1}}:\ \phi(n)\geq l_{m+1}\})
<(2\alpha(t_{m+1})+1-l_{m+1})(2\alpha(t_{m+1}))!/12
\leq 
$$
$$
\leq \alpha(t_{m+1})(2\alpha(t_{m+1}))!/6< \alpha(t_{m+1})(2\alpha(t_{m+1}))!/4< \card(\{n\in D_{t_{m+1}}:\ \phi(n)\geq l_{m+1}\}),
$$ a contradiction. 

Define $D=\{d_j:\ j\in\omega\}$. Then $D\in\Z_{f_M}$, since whenever $n\in [l_m, l_{m+1})$, we get 
$$
\frac{\card(D\cap n)}{f_M(n)}\leq
\frac{\card(D\cap l_{m+1})}{f_M(l_m)} \leq 
\frac{m+2}{f_M(l_m)}< \frac{1}{m} \to 0.
$$
However, $\phi^{-1}[D]\notin\Z_{f_K}$ since 
$$
\frac{\card(\phi^{-1}[D]\cap (2\alpha(t_m)+1)!-1))}{f_K((2\alpha(t_m)+1)!-1)}=
\frac{\card(\phi^{-1}[D]\cap (2\alpha(t_m)+1)!-1))}{(2\alpha(t_m))!} \geq
\frac{(2\alpha(t_m))!/12}{(2\alpha(t_m))!} =12
$$ for all $m\in\omega$.
\end{proof}

Recall that ideals $\I$ and $\J$ are isomorphic ($\I\cong\J$) if there is such a bijection $\phi\colon\omega\to\omega$ that $A\in\I\Longleftrightarrow\phi^{-1}[A]\in\J$ for all $A\subseteq\omega$. Since every pair of isomorphic ideals is comparable in the sense of Kat\v{e}tov order, we get the following.

\begin{cor}
\label{non-isomorphic}
There are $\ce$ many non-isomorphic ideals of the form $\Z_g$, for $g\in G$.
\end{cor}

Note also that Theorem \ref{antyincl} generalizes \cite[Theorem 2.7]{Den}, however, Theorem \ref{antikate} is not stronger than any of them. Indeed, \cite[Theorem 2.7]{Den} gives an example of such antichain of size $\ce$ that all its elements are incomparable with classical density ideal and Theorem \ref{antyincl} allows us to omit in such a way any family of cardinality $<\mathfrak{b}$, while Theorem \ref{antikate} does not provide us with an antichain of elements incomparable with the classical density ideal. Now we will prove that such generalization cannot be obtained as the poset $(\{\Z_g: g \in G\}, \leq_K)$ has the biggest element. We start with the following Lemma about inclusions: 

\begin{lem}
\label{liminf}
For $f \in G$ we have $\Z_f\subseteq\Z$ if and only if $\liminf_{n\to\infty}\frac{n}{f(n)}>0$.
\end{lem}

\begin{proof}
($\Rightarrow$): Suppose to the contrary that $\Z_f \subseteq \Z$, $\liminf_{n\to\infty}\frac{n}{f(n)}=0$, and find such an increasing sequence $(m_k)$ that $ m_k/f(m_k)<1/k$ and $m_{k+1}>2m_k$. Consider the set $A=\bigcup_{k\in\omega}[m_k,2m_k)$. We have $\card(A\cap 2m_k)/(2m_k)\geq 1/2$, thus $A \notin \Z$. On the other hand, for any $n \in \omega$ we may find such $k \in \omega$ that $m_k\leq n<m_{k+1}$, thus 
$$
\frac{\card(A\cap n)}{f(n)}\leq\frac{\card(A\cap n)}{f(m_k)}<\frac{m_k + \card(A \cap [m_k,2m_k))}{f(m_k)}\leq\frac{2m_k}{f(m_k)} \leq \frac{2m_k}{k m_k} = \frac{2}{k} \to 0.
$$ 
Hence, $A\in\Z_f$ -- a contradiction with $A \notin \Z_f$.

($\Leftarrow$): Fix $A\in\Z_f$, $\varepsilon>0$ and let $\delta:=\liminf_{n\to\infty}\frac{n}{f(n)}>0$. There is such $n_0$ that $n/f(n)>\delta/2$ for all $n>n_0$. There is also such $n_1$ that $\card(A\cap n)/f(n)<\varepsilon\delta/2$ for all $n>n_1$. Then for each $n>\max\{n_0,n_1\}$ we have 
$$
\frac{\card(A\cap n)}{n}=\frac{1}{\delta}\frac{\card(A\cap n)}{n/\delta}<\frac{1}{\delta}\frac{\card(A\cap n)}{f(n)/2}= \frac{2}{\delta}\frac{\card(A\cap n)}{f(n)}<\varepsilon,
$$
thus $A \in \Z$. 
\end{proof}
Recall that $\I \subseteq \J$ implies $\I \leq_K \J$ with identity function witnessing this fact. We are ready to prove the next Theorem.

\begin{thm}
\label{Zg-K-Z}
If $\I$ is a simple density ideal, then $\I\leq_{K}\Z$.
\end{thm}

\begin{proof}
Fix such $f \in H^\uparrow$ that $\I = \Z_f$. If $\liminf_{n\to\infty}\frac{n}{f(n)}>0$, then we are done by Lemma \ref{liminf}, so we may assume that $\liminf_{n\to\infty}\frac{n}{f(n)}=0$.\\
Set $\delta:=(\limsup_{n\to\infty}\frac{n}{f(n)})/2$ (recall that $\limsup_{n\to\infty}\frac{n}{f(n)}>0$ for each $f\in G$) and let $(k_m)$ be such an increasing sequence that $k_m\geq\delta f(k_m)$ and $\lim_{m\to\infty}k_m/k_{m+1}=0$. Now for any $n >k_0$ we may find such $m_n,j_n \in \omega$ that $n \in [k_{m_n}, k_{m_n +1})$ and $n \in [j_n k_{m_n}, (j_n+1)k_{m_n})$. Now we set 
$$\phi(n)=\left\{\begin{array}{ll}
n & \textrm{if } n\leq k_0,\\
n-j_n k_{m_n} & \textrm{otherwise.}
\end{array}\right.$$
We will show that $\phi$ is as required.\\
Fix $A\in\Z_f$ and $\varepsilon>0$. There is such $l_0\in\omega$ that $k_m/k_{m+1}<\varepsilon/3$ for all $m\geq l_0$. 
There is also such $l_1\in\omega$ that $\card(A\cap k_m)/f(k_m)<\varepsilon\delta/3$ for all $m\geq l_1$. 
Consider such $n\in\omega$ that $m_n>\max\{l_0,l_1\}$. Then
$$\frac{\card(\phi^{-1}[A]\cap n)}{n}\leq\frac{\card(\phi^{-1}[A]\cap n\setminus k_{m_n})}{n}+\frac{\card(\phi^{-1}[A]\cap k_{m_n}\setminus k_{m_n-1})}{n}+\frac{k_{m_n-1}}{n} \leq$$
$$\leq\frac{\card(A\cap k_{m_n})}{n}j_n+\frac{\card(A\cap k_{m_n-1})}{n}j_{k_{m_n}-1}+\frac{k_{m_n-1}}{n}\leq$$
$$
 \leq \frac{\card(A\cap k_{m_n})}{f(k_{m_n})}\frac{f(k_{m_n})}{k_{m_n}}\frac{k_{m_n}}{n}j_n+\frac{\card(A\cap k_{m_n-1})}{f(k_{m_n-1})}\frac{f(k_{m_n-1})}{k_{m_n-1}}\frac{k_{m_n-1}}{k_{m_n}-1}j_{k_{m_n}-1}+\frac{k_{m_n-1}}{k_{m_n}}<
 $$
$$
<\frac{\varepsilon\delta}{3}\frac{1}{\delta}1+\frac{\varepsilon\delta}{3}\frac{1}{\delta}1+\frac{\varepsilon}{3}=\varepsilon.$$
Therefore, $\phi^{-1}[A]\in\Z$.
\end{proof}
\begin{Rem}
There are many other preorders on the set of ideals on $\omega$. One of them is the Kat\v{e}tov-Blass preorder. We say that $\I \leq_{KB} \J$ if there exists such $\phi: \omega \to \omega$ that $\phi^{-1}(A) \in J $ whenever $A \in \I$, and $\phi$ is finite-to-one, i.e., preimages of singletons are finite. Since $\I \leq_{KB} \J \implies \I \leq_K \J$, Theorem \ref{antikate} holds true for $\leq_{KB}$. In fact, Theorem \ref{Zg-K-Z} also holds -- one may provide a similar construction with $\phi([k_m, k_{m+1}]) \subset [k_{m-1} , k_m)$.
\end{Rem}
\begin{Rem}
Observe that implication from Theorem \ref{Zg-K-Z} cannot be reversed -- it is known that the ideal $\I_{1/n}=\{A\subseteq\omega:\ \sum_{a\in A}1/a<\infty\}$ is contained in $\Z$ (hence, $\I_{1/n}\leq_{KB} \Z$ and $\I_{1/n}\leq_K \Z$), but it is not a simple density ideal (it is $\mathbf{F}_{\sigma}$ by \cite[Example 1.2.3.(c)]{Farah} and all simple density ideals are not $\mathbf{F}_{\sigma}$ by \cite[Corollary 3.5.]{Den} or Proposition \ref{Fsigma}).
\end{Rem}

\section{simple density ideals, Farah's density ideals and Erd\H{o}s-Ulam ideals}
\label{sekGD+FD+EU}
In \cite[Theorem 3.2]{Den} it has been proved that each simple density ideal is a density ideal in the sense of Farah. Within the second class we may distinguish a family of Erd\H{o}s-Ulam ideals (EU ideals in short). This Section is devoted to the investigation of the inclusions between those three families, so we start with recalling the necessary definitions. 

We say that an ideal $\I$ on $\omega$ is a density ideal in the sense of Farah (or Farah's density ideal) if there exists such a sequence $(\mu_n)$ of measures on $\omega$ that their supports are finite, pairwise disjoint and 
$$
\I:= \{ A \subset \omega: \lim_{n \to \infty} \sup_{m \in \omega} \mu_m ( A \setminus n) = 0 \}.
$$
In such a case we say that $\I$ is generated by the sequence $(\mu_n)$. For more information on ideals defined by measures and submeasures see e.g. \cite{Farah} and \cite{Solecki}.\\
The second important class of ideals are Erd\H{o}s-Ulam ideal (or EU ideals). Namely, those are ideals of the form
$$
 \mathcal{EU}_f=\left\{ A \subset \omega\colon \lim_{n \to \infty} \frac{\sum_{i \in n \cap A}f(i)}{{\sum}_{i\in n}f(i)}=0\right\},
$$
where $f\colon \omega \to[0, \infty)$ satisfies ${\sum}_{n\in\omega}f(n) = \infty$. EU ideals were introduced in \cite{Just} by Just and Krawczyk, where they solved a question raised by Erd\H{o}s and Ulam. The following Theorem outlines the connections between those two classes of ideals.
\begin{thm}\cite[Theorem 1.13.3]{Farah}  \label{density-is-eu}
\begin{itemize}
\item[(a)] Each EU ideal $\I$ is a density ideal in the sense of Farah and there exists a sequence of probability measures which generates $\I$.
\item[(b)] Let $\I$ be a density ideal generated by a sequence of measures $(\mu_n)$. Then $\I$ is an Erd\H{o}s-Ulam ideal if and only if the following conditions hold:
\begin{itemize}
\item[(D1)] $\sup_{n\in\omega} \mu_n(\omega) < \infty$,
\item[(D2)] $\lim_{n\to\infty} \sup_{i\in\omega} \mu_n(\{i\}) = 0$,
\item[(D3)] $\limsup_{n\to\infty} \mu_n(\omega) > 0$. \qed
\end{itemize}
\end{itemize}
\end{thm}
It is worth to mention that Farah's density ideal $\I$ is tall (i.e., each infinite set in $\omega$ contains an infinite subset which belongs to $\I$) if and only if the condition $(D2)$ holds (for a detailed explanation see \cite[Proposition 3.4]{Den}). Thus, any sequence of measures which generates a simple density ideal satisfies this condition. Moreover, the condition $(D3)$ is equivalent to $\omega \notin \I$, so it also must be satisfied whenever $\I$ is a simple density ideal. Note that a similar discussion, in a bit more detailed way, may be found in \cite[Section 3]{Den}, however, we repeated it for the convenience of the Reader.

We can modify \cite[Theorem~3.2]{Den} a bit to obtain the following.

\begin{prop}\label{densities}(Essentially \cite[Theorem~3.2]{Den})
Suppose that $g\in H$. Let $n_0=0$ and $n_k=\min\{n\in\omega:g(n)\geq 2^k\}$ for $k>0$. Then $\Z_g$ is a density ideal generated by the sequence $(\mu_k)_{k\in\omega}$ of measures given by
$$\mu_k(A)=\frac{\card(A\cap [n_k,n_{k+1}))}{g(n_k)} .$$ \qed
\end{prop}
The only difference between this Proposition and Theorem~3.2 from \cite{Den} is that originally in the sequence instead of $n_k$, as defined above, there was $m_k=\min\{n\in\omega:g(n)\geq 2g(m_{k-1})\}$. It is not difficult to check that  every $[n_k,n_{k+1})$ is split between at most two neighboring intervals $[m_i,m_{i+1})$ and every interval $[m_i,m_{i+1})$ is split between at most two neighboring intervals $[n_k,n_{k+1})$. Therefore, when $A\subseteq\omega$,  for every $k\in\omega$ there is $i\in\omega$ such that  $\card(A \cap [n_k,n_{k+1}))\leq 2\card(A \cap [m_i,m_{i+1}))$ and vice versa. There is no consequence of the fact that some intervals $[n_k,n_{k+1})$ may be empty.

Moreover, it is easy to see that a density ideal generated by the sequence $(\phi_k)_{k\in\omega}$ of measures given by $\phi_k(A)=\frac{\card(A\cap [a_k,a_{k+1}))}{g(a_k)}$ for some increasing sequence of natural numbers $(a_k)_{k\in\omega}$, is the same as the one generated by the sequence $(\psi_k)_{k\in\omega}$ of measures given by $\psi_k(A)=\frac{\card(A\cap [a_k,a_{k+1}))}{2^{l_k}} $, where $l_k=\min\{n\in\omega: g(n_l)\geq 2^n  \}$. 

By the argument about splitting intervals, given above, we clearly get that the ideals generated by measures  $\nu_k(A)=\frac{\card(A\cap [n_k,n_{k+1}))}{2^k}$ and $\xi_k(A)=\frac{\card(A\cap [m_k,m_{k+1}))}{2^{l_k}}$ are the same.

\begin{prop}
\label{EU}
Let $g\in H^\uparrow$. Then the ideal $\Z_g$ is an Erd\H{o}s-Ulam ideal if and only if the sequence $(\card(g^{-1}([2^n,2^{n+1})))/2^n)_{n\in\omega}$ is bounded. 
\end{prop}
\begin{proof}
By Theorem~\ref{density-is-eu} and the comments below it, we only need to check equivalence with the condition $(D1)$.
Observe that 
$$\mu_k(\omega)=\frac{\card(\omega\cap [n_k,n_{k+1}))}{g(n_k)}=\frac{\card(g^{-1}([2^k,2^{k+1})))}{g(n_k)}$$
for all $k \in \omega$. Thus, $\sup_{n\in\omega} \mu_n(\omega) < \infty$ if and only if $\sup_{n\in\omega} \card(g^{-1}([2^n,2^{n+1})))/2^n < \infty$.
\end{proof}

Now we will give another characterization of when $\Z_f$ is an Erd\H{o}s-Ulam ideal. The following condition is more similar to the one from Theorem \ref{moce} and will allow us to compare cardinality of $S_f$ with the fact that $\Z_f$ is an Erd\H{o}s-Ulam ideal. Recall that "for almost all $n\in\omega$" means "for all but finitely many $n\in\omega$". 

\begin{prop}\label{EU2}
Let $f\in H^\uparrow$. Then the ideal $\Z_f$ is an Erd\H{o}s-Ulam ideal if and only if for each $M>0$ there is such $L>0$ that for almost all $n\in\omega$ we have
$$\frac{f(n+\lfloor Lf(n)\rfloor)}{f(n)}>M.$$
\end{prop}
\begin{proof}
($\Rightarrow$):
Let $M>0$. There is such $k_0\in\omega$ that $2^{k_0}>M$. Since $\Z_f$ is an Erd\H{o}s-Ulam ideal, by Proposition \ref{EU} there is such $k_1\in\omega$ that $\card(f^{-1}[[2^n,2^{n+1})])\leq 2^{n+k_1}$ for each $n\in\omega$. Let $L=2^{k_0+k_1 +1}$. We will show that $\frac{f(n+\lfloor Lf(n)\rfloor)}{f(n)}>M$ whenever $f(n)\geq 1$ (it is enough since $f(n) \to \infty)$. Fix any $n$ satisfying this condition and let $k\in\omega$ be given by $2^k\leq f(n)<2^{k+1}$. Then 
$f(n+\lfloor Lf(n)\rfloor)\geq f(n+2^{k_0+k+k_1 +1})\geq f(2^{k_0+k+k_1 +1})$. Moreover,
$$
\card(f^{-1}([0,2^{k_0 + k + 1}))) = \sum_{i=0}^{k_0 +k} \card\left(f^{-1}([2^i, 2^{i+1}))\right) \leq \sum_{i=0}^{k_0 +k} 2^{k_1 + i} =
$$
$$
= 2^{k_1} \frac{1-2^{k_0 +k+1}}{1-2} = 2^{k_0 + k_1 + k+1}- 2^{k_1} < 2^{k_0 + k_1 + k+1},
$$
thus $f(2^{k_0+k+k_1 +1 })\geq 2^{k_0+k+1}$ and hence 
$$
\frac{f(n+\lfloor Lf(n)\rfloor)}{f(n)}\geq\frac{2^{k_0+k+1}}{2^{k+1}}=2^{k_0}>M,
$$
so we are done. \\
($\Leftarrow$):
Let $M=2$. Then there are such $L>0$ and $n_0\in\omega$ that $\frac{f(n+\lfloor Lf(n)\rfloor)}{f(n)}>M$ for all $n>n_0$. Let $\delta=2L$. We will show that $\card(f^{-1}[[2^n,2^{n+1})])\leq 2^{n}\delta$ for all such $n\in\omega$ that $f(n_0)<2^n$. Fix such $n$ and let $k=\min f^{-1}[[2^n,2^{n+1})]$ (the case $f^{-1}[[2^n,2^{n+1})]=\emptyset$ is trivial). Note that $k>n_0$ and thus 
$$
2<\frac{f(k+\lfloor Lf(k)\rfloor)}{f(k)}\leq
\frac{f(k+\lfloor L 2^{n+1}\rfloor)}{2^n} \leq \frac{f(k+\lfloor 2^n\delta\rfloor)}{2^n}.
$$
Hence, $f(k+\lfloor 2^n\delta\rfloor)>2^{n+1}$ and $\card(f^{-1}[[2^n,2^{n+1})])\leq \lfloor 2^n\delta\rfloor \leq 2^n \delta$. Thus, we obtain that the sequence $(\card(f^{-1}[[2^n,2^{n+1})])/2^{n})_{n\in\omega}$ is bounded by $\delta$ and we are done by Proposition \ref{EU}.
\end{proof}
\begin{Rem}
The condition obtained in Proposition \ref{EU2} is equivalent to the analogous one in which "for almost all $n\in\omega$" is replaced by "for all $n\in\omega$". 
\end{Rem}
Now we will give some examples of ideals, which will show that there is no dependence between the cardinality of $S_f$ and the fact that $\Z_f$ is an Erd\H{o}s-Ulam ideal. Recall that, by Theorem \ref{moce}, the condition $\card(S_f)=1$ is equivalent to the following one: there are such $M>0$ and $\varepsilon>0$ that for almost all $n\in\omega$ we have 
$$
\frac{f(n+\lfloor\varepsilon f(n)\rfloor)}{f(n)}\leq M.
$$

We start by observing that the classical density ideal $\Z$ is an EU ideal and $\card(S_{\text{id}})=1$. The first fact is well-known (it may also be easily proved using Proposition \ref{EU2}) and the second one follows from
$$
\frac{\text{id}(n + \lfloor \varepsilon \cdot\text{id}(n) \rfloor}{\text{id}(n)} = 
\frac{n + \lfloor \varepsilon n \rfloor }{n} \leq 1 + \varepsilon,
$$
which holds for any $\varepsilon>0$ (so it is enough to fix $\varepsilon>0$, $M> 1 + \varepsilon$, and use Theorem \ref{moce}).

Now we give an example of such $f\in H^\uparrow$ that $\Z_f$ is an Erd\H{o}s-Ulam ideal and $\card(S_f)=\ce$. 
\begin{ex}
Let us set $f(0)=0$ and $f(n) = (k+1)!$ whenever $n \in [k!, (k+1)!)$. Clearly, $f\in H^{\uparrow}$ (as $(k!-1)/f(k!-1) = (k!-1)/k! \to 1 \neq 0$). What is more, $\Z_f$ is as required.
\begin{itemize}
 \item Fix $ \varepsilon >0$. Then for such $k$ that $\varepsilon f(k!-1) \geq 1$, we have
$$
\frac{f(k!-1 + \lfloor \varepsilon f(k!-1)\rfloor) }{f(k!-1)} = 
\frac{f(k!-1 + \lfloor \varepsilon f(k!-1)\rfloor) }{k!} \geq
\frac{f(k!)}{k!} = \frac{(k+1)!}{k!}=k+1 \to \infty,
$$ 
hence, by Theorem \ref{moce}, $\card(S_f) = \ce$.
\item Take any $M>0$, set $L=1$, and for $n \in \omega$ find such $k \in \omega$ that $n \in [k! , (k+1)!)$. Then 
$$
\frac{f(n + \lfloor L f(n) \rfloor)}{f(n)} = \frac{f(n + (k+1)!)}{(k+1)!} \geq \frac{(k+2)!}{(k+1)!} = k+2 \to \infty,
$$
so, by Proposition \ref{EU2}, the ideal $\Z_f$ is EU. 
\end{itemize}
\end{ex}

The next example is such that $\Z_f$ is not an Erd\H{o}s-Ulam ideal and $\card(S_f)=\ce$. 
\begin{ex}
Define $m_0=0$, $f(m_0)=1$, $m_{k}=\sum_{i=0}^{k} i!$, and $f(m_{k})=k!$ for $k \geq 1$. Let $f(n)=k!$ for all $m_k<n<m_{k+1}$. Clearly, $f\in H^\uparrow$. Moreover, $\Z_f$ is as required.
\begin{itemize}
 \item Fix $\varepsilon>0$. Then whenever $\varepsilon f(m_{k+1}-1) \geq 1$, we get
 $$
 \frac{f(m_{k+1}-1 + \lfloor \varepsilon f(m_{k+1} -1)\rfloor)}{f(m_{k+1}-1)}=
 \frac{f(m_{k+1}-1 + \lfloor \varepsilon f(m_{k+1} -1)\rfloor)}{f(m_k)} \geq \frac{f(m_{k+1})}{k!} = ~\frac{(k+1)!}{k!} = k+1 \to \infty,
 $$
 so, by Theorem \ref{moce}, we get $\card(S_f)= \ce$.
\item For an arbitrary $L>0$ and $k>L-1$ we get 
$$
\frac{f(m_k + \lfloor L f(m_k)\rfloor)}{f(m_k)}= 
\frac{f(m_k + \lfloor L k! \rfloor)}{f(m_k)}=
\frac{f(m_k)}{f(m_k)}=1,
$$
since $m_k + L k! <m_k + (k+1)!= m_{k+1}$. Thus, by Proposition \ref{EU2}, we see that $\Z_f$ is not an EU ideal. 
\end{itemize}
\end{ex}

Finally, we give an example of such $f\in H^\uparrow$ that $\Z_f$ is not an Erd\H{o}s-Ulam ideal and $\card(S_f)=1$. 
\begin{ex}
Define $m_0=0$, $f(m_0)=1$, $m_{k+1}=m_k+(k+1)f(m_k)$, and $f(m_{k+1})=2^{k+1}$. Let $f(n)=f(m_k)$ for all $m_k<n<m_{k+1}$. It is easy to see that $f\in H^\uparrow$. Moreover, $\Z_f$ is as required.
\begin{itemize}
 \item Set $\varepsilon=1$, $M=2$, and for $n \in \omega$ find such $k \in \omega$ that $n \in [m_k, m_{k+1})$. Note that 
$$m_{k+2}-m_{k+1}=(k+2)f(m_{k+1})>f(m_k)=f(n),$$ thus $n+f(n)<m_{k+2}$, and
 $$
\frac{f(n + \lfloor \varepsilon f(n)\rfloor)}{f(n)}=
\frac{f(n + f(n))}{f(n)} = \frac{f(n+f(n))}{f(m_k)} \leq \frac{f(m_{k+1})}{f(m_k)} = 2 = M,
 $$
so, by Theorem \ref{moce}, we get $\card(S_f)= 1$.
\item For an arbitrary $L>0$ and $k>L-1$ we get 
$$
\frac{f(m_k + \lfloor L f(m_k)\rfloor)}{f(m_k)}= 
\frac{f(m_k)}{f(m_k)}=1,
$$
since $m_k + L f(m_k ) < m_k + (k+1)f(m_k)= m_{k+1}$. Thus, by Proposition \ref{EU2}, we see that $\Z_f$ is not an EU ideal. 
\end{itemize}
\end{ex}
So far in this Section we were dealing with the problem regarding when a simple density ideal is an EU ideal. A forthcoming paper \cite{EUvsGD} characterizes EU ideals which are simple density ideals. Now we move to another question: when a Farah's density ideal is a simple density ideal? Note that, as mentioned earlier, such ideal has to be tall. We give only a partial answer to the above question, while the full characterization of this problem remains open. 

For a measure $\mu$, by $\supp(\mu)$ we denote its support, i.e., $\supp(\mu)=\{n\in\omega:\mu(\{n\})>0\}$.

\begin{prop}\label{nic}
Let us assume that $(\mu_n)_{n\in\omega}$ is such a sequence of measures that:
\begin{itemize}
\item[(i)] $I_n := \supp(\mu_n)$ is an interval for all $n$, and $\max (I_n) < \min(I_{n+1})$;
\item[(ii)] each $\mu_n$ is uniformly distributed on its support $I_n$;
\item[(iii)] $a_n \min (I_n) \to 0$ and $a_n>0$ for each $n$, where $a_n = \mu_n(\{ m\})$ for $m \in I_n$ (by $(ii)$, it doesn't matter which $m$ we pick);
\item[(iv)] $a_n \to 0$ (it is equivalent to the condition $(D2)$ and tallness of the generated ideal);
\item[(v)] $\mu_n(\omega) \not\to 0$ (it is equivalent to the condition $(D3)$ and $\omega$ not belonging to the generated ideal); 
\item[(vi)] $(a_n)_{n\in\omega}$ is a nonincreasing sequence.
\end{itemize} 
Then an ideal $\I$ generated by $(\mu_n)_{n\in\omega}$ is a simple density ideal generated by a function $g$ given by $g(k) = 1/a_0$ for $k \leq \max(I_0)$, and $g(k)=1/a_{n+1}$ for $k \in (\max(I_n), \max (I_{n+1})]$.
\end{prop}
\begin{proof}
We will prove that $g \in G$ in the last part of the proof. Firstly, we prove that $\Z_g \subset \I$, so let us fix $A \in \Z_g$. It is enough to show that $\mu_n (A) \to 0$. We have
$$
\mu_n(A) = \card(A \cap I_n) a_n \leq \frac{\card(A \cap \max(I_n))}{1/a_n}=\frac{\card(A \cap \max(I_n))}{g(\max(I_n))} \to 0,
$$
so we are done. \\
Now we move to proving $\I \subset \Z_g$, so let us fix $A \in \I$. For any $m > \max(I_1)$ we may find such $n \in \omega$ that $m \in (\max(I_n), \max(I_{n+1})]$, and then
$$
\frac{\card(A \cap m)}{g(m)} \leq \left(\min(I_{n}) + \card(A \cap I_{n})\right)a_{n} = \min(I_{n})a_{n} + \card(A \cap I_{n})a_{n}=  \min(I_{n})a_{n}+ \mu_n(A) \to 0,
$$
so we are done. We proved that $\mu_n (A) \to 0 \Leftrightarrow \card(A \cap n)/g(n) \to 0$, so it is enough to show that $g \in G $ now. Clearly, the values of $G$ are nonnegative and $g(n)\to \infty$ as $a_n \to 0$, so it is enough to show that $\frac{n}{g(n)} \not\to 0$. This follows immediately from the fact that $\omega \notin \I$ -- indeed,
$$
\frac{n}{g(n)} \to 0 \Leftrightarrow \frac{\card(\omega \cap n)}{g(n)} \to 0 \Leftrightarrow \mu_n(\omega) \to 0 \Leftrightarrow \omega \in \I,
$$
so we conclude that $g \in G$. 
\end{proof}
Now we need an idea of increasing-invariant ideals, introduced in the forthcoming paper \cite{EUvsGD} and inspired by \cite{Inv} (see also \cite[Example 3.2]{Den}). It will also be important in the next Section.
\begin{df}
We say that an ideal $\I$ on $\omega$ is \emph{increasing-invariant} if for every $B\in\I$ and $C\subseteq\omega$ satisfying $\card(C\cap n)\leq\card(B\cap n)$ for all $n$, we have $C\in\I$. Equivalently, for any increasing injection $f :\omega \to \omega$ and $B \in \I$, we have $f(B) \in \I$.
\end{df}
\begin{Rem}
Sometimes, an ideal $\I$ is called shift-invariant if $A\in\I$ implies $\{a+k:\ a\in A\}\cap\omega\in\I$ for every $k\in\mathbb{Z}$ (note that here the shift is the same for all points from $A$, while our notion of increasing-invariance allows different shifts for different points).
\end{Rem}
The next result establishes a connection of the above notion with simple density ideals. Actually, this connection is much deeper, as shown in the forthcoming \cite{EUvsGD}. We omit the proof, since it is obvious (see also \cite[Section 4]{Inv}).
\begin{prop}
Each simple density ideal is increasing-invariant. \qed
\end{prop}
We conclude this Section with some easy examples showing that we cannot simply omit any assumption in Proposition \ref{nic}. This is obvious for the conditions (iv) and (v), since since they are equivalent to tallness of $\I$ and $\omega \not\in \I$, respectively. Furthermore, without the condition (vi) it would be easy to construct a density ideal that is not increasing-invariant. Now we deal with the assumption (ii) (see also \cite[Example 3.2]{Den}).
\begin{ex}
Given an $n \in \omega$, we define $\mu_n$ as a probabilistic measure for which $\supp\mu_n=[2^n,2^{n+1})$, 
\newline 
$\mu_n([2^n,2^n+n))=1-\frac{1}{2^n}$, and $\mu_n$ is uniformly distributed on $[2^n,2^n+n)$ and $[2^n+n,2^{n+1}).$
Let us denote by $\I$ the EU ideal generated by $(\mu_n)_{n\in\omega}$, and put 
$$A=\bigcup_{n \geq 1} [2^n,2^n+n) \mbox{ and } B=\bigcup_{n \geq 1} [2^n-n,2^n).$$
Then, $A \notin \I$ and $B \in \I$. Indeed, $\mu_n(A) =1-\frac{1}{2^n} \not\to 0$ and $\mu_n(B) \leq \frac{1}{2^n}\to0$. If we assume that $\I$ is a simple density ideal, it has to be increasing-invariant. Thus, we get $A \in \I$, since $A$ is the image of $B$ through an increasing injection -- a contradiction.
\end{ex}
Now we move to the assumption (iii).
\begin{ex}
Let us take any sequence of measures $(\mu_n)_{n\in\omega}$ satisfying the assumptions of Proposition \ref{nic}. Define a sequence of measures $(\mu'_n)_{n\in\omega}$ in such a way that $\mu'_n$ is obtained by translating $\supp(\mu_n)$, and $\min(\supp(\mu'_{n+1}))-\max(\supp(\mu'_n)) > \card(\supp(\mu_{n+1}))$. Conditions (i), (ii), (iv), (v) and (vi) are clearly satisfied. However, an ideal $\I'$ generated by $(\mu'_n)$ is not increasing-invariant. Indeed, set  
$$
A=\bigcup_{n \geq 1} \supp(\mu'_n) \mbox{ and } B=\bigcup_{n \geq 1} [\min(\supp(\mu'_{n+1}))- \card(\supp(\mu'_{n+1})),\min(\supp(\mu'_{n+1}))).
$$
Then, clearly, $\mu'_n (A) = \mu_n (\omega) \not\to 0$ and $\mu'_n (B)= 0$ for all $n \in \omega$. However, $A$ is covered by an image of $B$ through an increasing injection $f$ given by $f(n) = n + \card(\supp(\mu'_{n+1}))$ whenever $n \in [\max(\supp(\mu'_n)), \max(\supp(\mu'_{n+1}))$, thus $\I'$ is not increasing-invariant and -- as a consequence -- also not a~simple density ideal. 
\end{ex}
One may say that the above Example deals only with a half of the condition (iii), however, the assumption $a_n>0$ is in fact only a notation -- by (v), we see that there are infinitely many such $n$ that $a_n>0$, thus we may just omit in the sequence $(\mu_n)$ these measures which are constantly equal to zero and the generated ideal remains the same.\\
Finally, we deal with the assumption (i).
\begin{ex}
Let $\mu_n$ be a probability measure uniformly distributed on a set $2\omega \cap [2n!, 2(n+1)!)$. The assumptions (ii), (iv), (v) and (vi) are clearly satisfied. In order to prove that (iii) holds, note that $a_n = 2/(2(n+1)! - 2n! ) = \frac{1}{n!n}$, and hence 
$$
a_n \min(\supp(\mu_n)) = \frac{1}{n!n} 2n!= \frac{2}{n} \to 0,
$$
so the condition (iii) is satisfied. Now we set $A:= 2\omega \setminus \{0\}$ and $B = \omega \setminus 2\omega$. Then, clearly, $\mu_n(A) = 1 $ and $\mu_n(B) = 0$ for $n \geq 1$, thus $A \notin \I$ and $B \in \I$. However, $A= B+1$, thus $\I$ is not increasing-invariant.
\end{ex}
\section{simple density ideals and homogeneity}
\label{sekhomogeneity}

Recall that if $A\notin\I$, then $\I|A=\{A\cap B: B\in\I\}$ is an ideal on $A$ (called the restriction of $\I$ to $A$) and that two ideals $\I$ and $\J$ are isomorphic ($\I\cong\J$) if there is such a bijection $\phi\colon\omega\to\omega$ that $A\in\I\Longleftrightarrow\phi^{-1}[A]\in\J$ for $A\subseteq\omega$. 

In this short Section we take a look on simple density ideals in the context of ideas which have recently appeared in the papers \cite{Inv} and \cite{kwetry}. In particular, we give a partial solution to \cite[Problem 5.8]{kwetry}: characterize ideals $\I$ such that if $\mathcal{I}|A\cong\mathcal{I}$, then it is witnessed by the increasing enumeration of $A$.

\begin{df}
Let $\I$ be an ideal on $\omega$. Then
$$H(\I)=\left\{A\subseteq\omega:\ \mathcal{I}|A\mbox{ is isomorphic to }\mathcal{I}\right\}$$
is called the \emph{homogeneity family} of the ideal $\mathcal{I}$. 
\end{df}

The next Theorem shows that all increasing-invariant ideals, in particular simple density ideals, possess the property mentioned in the above problem. 
\begin{thm}\label{homogeneity}
Let $\I$ be an increasing-invariant ideal. If $A\in H(\I)$ and $\{a_0,a_1\ldots\}$ is an increasing enumeration of $A$, then the function $f\colon\omega\to A$ given by $f(n)=a_n$ witnesses that $\I|A\cong\I$.
\end{thm}
\begin{proof}
Take an increasing-invariant ideal $\I$ and a set $A\in H(\I)$. Suppose that $f(n)=a_n$ is not an isomorphism between  $\I|A$ and  $\I$. Let $g:\omega \rightarrow A$ be such an isomorphism. It is easy to see that for each  increasing-invariant ideal $\I$ every increasing function is such that $f[B]\in\I$ for all $B\in\I$, so there has to be a set $B\notin\I$ such that $f[B]\in\I$. Let $\{b_0,b_1,\ldots \}$ be an increasing enumeration of $B$ and let $\{c_0,c_1,\ldots\}$  be an increasing enumeration of $f[B]$.

We define inductively a sequence $(k_i^j)$ for $i,j\in\omega$, $i\leq j$. Let $k_0^0$ be a natural number such  that  $k_0^0\leq b_0$ and $g(k_0^0)\geq c_0$. We can find such an element, because $f$ is an increasing enumeration of $A$,  hence $\card(A \cap c_0)=b_0$. 

Assume that we have defined  $(k_i^j)$ for $i,j<n$, $i\leq j$. Now we choose the elements $k_0^n,\ldots , k_n^n$ in such a way that $\{k_0^{n-1},\ldots, k_{n-1}^{n-1}\}\subseteq\{k_0^n,\ldots,k_n^n\}$, $k_i^n\leq b_n$, and $g(k_i^n)\geq c_i$ for every $i\leq n$.

Define the set $D=\bigcup_{ i,j\in\omega, i\leq j}\{k_i^j\}$. Then, for each $n\in\omega$ we have $\card(D\cap n)\geq \card(B\cap n)$, thus $D\not\in\I$. On the other hand, for every $n\in\omega$  we have  $\card(g[D]\cap n)\leq \card(f[B]\cap n)$, hence $g[D]\in\I$. A contradiction with the fact that $g$ is an isomorphism.
\end{proof}

\begin{Rem}
Not all Erd\H{o}s-Ulam ideals have the property from Theorem \ref{homogeneity} -- it is easy to see that if $f(n)=1$ for odd $n$, and $f(n)=0$ for even $n$, then $\mathcal{EU}_f\upharpoonright(\omega\setminus\{0\})$ is isomorphic to $\mathcal{EU}_f$, but the increasing enumeration of $\omega\setminus\{0\}$ is not an isomorphism.
\end{Rem}

\begin{Rem}
Theorem \ref{homogeneity} in the case of the classical density ideal $\Z$ has been known earlier (cf. \cite[Theorem 5.6]{kwetry}). However, $\Z$ has been the only known example. 
\end{Rem}

By Corollary \ref{non-isomorphic} we get the following.

\begin{cor}
There are $\ce$ many non-isomorphic ideals $\I$ such that if $\mathcal{I}|A\cong\mathcal{I}$, then it is witnessed by the increasing enumeration of $A$.
\end{cor}

We end our paper with an example of a simple density ideal $\Z_g$ which is anti-homogeneous, i.e., $H(\Z_g)$ is equal to $\{\omega\setminus A:A\in\Z_g\}$ (the filter dual to the ideal $\Z_g$).

\begin{ex}
Let $(I_n)$ be the sequence of consecutive intervals such that the length of each $I_n$ is $n!$. Let also $(\mu_n)$ be a sequence of measures on $\omega$, given by:
$$\mu_n(k)=\left\{\begin{array}{ll}
\frac{1}{n!} & \textrm{if } k\in I_n,\\
0 & \textrm{otherwise.}
\end{array}\right.$$
Consider the density ideal $\I$ generated by $(\mu_n)$. In \cite[Theorem 4.8]{kwetry} it was shown that this is an anti-homogeneous EU ideal. It can also be used as a "base" ideal that can be used for creating, after modifications, $\ce$ many non-isomorphic Boolean algebras of the form $\pt(\omega)/\I$ (cf. \cite{oliver}). Moreover, by Proposition \ref{nic}, this is also a simple density ideal. Indeed, condition $(iii)$, is satisfied, as $$a_n\min(I_n)=\frac{\sum_{i<n}i!}{n!}\leq \frac{(n-1)!+(n-2)(n-2)!}{n!}\to 0,$$
and all other conditions are trivial.
\end{ex}

\section*{Acknowledgment}

The second author was supported by the University of Silesia Mathematics Department (Iterative Functional Equations and Real Analysis program).

\end{document}